\documentclass[a4,amstex,11pt]{article}
\usepackage{amsmath,wrapfig}
\usepackage{amssymb,bm,epic}
\usepackage[mathscr]{eucal}
\usepackage{amsthm, amsfonts, latexsym,enumerate}
\usepackage{graphicx}
\usepackage{setspace}
\makeatletter
\def\fixedfigure{\def\@captype{figure}}
\def\fixedtable{\def\@captype{table}}
\numberwithin{equation}{section}
\theoremstyle{definition}

\newtheorem{definition}{Definition}[section]
\newtheorem{theorem}[definition]{Theorem}
\newtheorem{proposition}[definition]{Proposition}
\newtheorem{lemma}[definition]{Lemma}

\theoremstyle{definition}
\newtheorem{remark}[definition]{Remark}
\newtheorem{example}[definition]{Example}
\title{Independence and orthogonality of algebraic eigenvectors over the max-plus algebra
}

\author{
Yuki Nishida\footnote{Organization for Research Initiatives and Development, Doshisha University,
1-3 Tatara-Miyakodani, Kyotanabe, 610-0394 Japan.
}
\and 
Sennosuke Watanabe\footnote{Department of Informatics, The University of Fukuchiyama,
3370 Azahori, Fukuchiyama, 620-0886, Japan}
\and
Yoshihide watanabe\footnote{Department of Mathematical Sciences, Doshisha University,
1-3 Tatara-Miyakodani, Kyotanabe, 610-0394 Japan.
}}
\date{October 4, 2021}
\begin{document}
\maketitle

% =====================================

%\renewcommand{\abstractname}{\large \textbf{Abstract}}
\begin{abstract}
The max-plus algebra $\mathbb{R}\cup \{-\infty \}$ is a semiring with the two operations: 
addition $a \oplus b := \max(a,b)$ and multiplication $a \otimes b := a + b$.
Roots of the characteristic polynomial of a max-plus matrix are called  algebraic eigenvalues.
Recently, algebraic eigenvectors with respect to algebraic eigenvalues were introduced as a generalized concept of eigenvectors.
In this paper, we present properties of algebraic eigenvectors analogous to those of eigenvectors in the conventional linear algebra.
First, we prove that for generic matrices algebraic eigenvectors with respect to distinct algebraic eigenvalues are linearly independent.
We further prove that for symmetric matrices algebraic eigenvectors with respect to 
distinct algebraic eigenvalues are orthogonal to each other.
\renewcommand{\thefootnote}{\fnsymbol{footnote}}
{}\footnotemark[0]
\footnotetext[0]{Email: ynishida.cyjc1901@gmail.com (Yuki NISHIDA)}
\end{abstract}

% =====================================


{\it 2010MSC}: 15A16, 15A80

\section{Introduction}
\label{intro}
The max-plus algebra $\mathbb{R}\cup \{-\infty \}$ is a semiring with the two operations: 
addition $a \oplus b := \max(a,b)$ and multiplication $a \otimes b := a + b$.
The max-plus algebra has its origin in steelworks~\cite{Green1960, Green1962}.
It has a wide range of applications in various fields of science and engineering, such as control theory and
scheduling of railway systems~\cite{Baccelli1992, Heidergott2005}. 
\par
The max-plus eigenvalue problem is an active research subject in the max-plus linear algebra. 
For a matrix $A \in \mathbb{R}_{\max}^{n \times n}$, a scalar $\lambda$ is called an eigenvalue of $A$
if there exists a nontrivial vector $\bm{x}$, called an eigenvector, satisfying $A \otimes \bm{x} = \lambda \otimes \bm{x}$.
Eigenvalues and eigenvectors are characterized in terms of graph theory.
A max-plus square matrix $A$ is associated with the weighted digraph $G(A)$ whose weighted adjacent matrix is $A$.
Then, the maximum eigenvalue is the maximum value of the average weights of circuits in $G(A)$.
There is an algorithm to find all eigenvalues and eigenvectors in polynomial time~\cite{Butkovic2010b}.
It is notable that max-plus matrices have a few eigenvalues.
In particular, an irreducible max-plus matrix has exactly one eigenvalue.
By contrast, the characteristic polynomial of an $n \times n$ max-plus matrix admits exactly $n$ roots counting multiplicities.
These roots are called algebraic eigenvalues~\cite{Akian2006}.
The maximum algebraic eigenvalue is always the maximum eigenvalue of the matrix~\cite{Green1983}, 
but other algebraic eigenvalues are not eigenvalues in general, i.e., they do not have corresponding eigenvectors.
\par
Recently, the authors have clarified the role of the roots of the characteristic polynomial $\chi_{A}(t)$
of a max-plus matrix $A$~\cite{Nishida2020}.
Each coefficient of the characteristic polynomial of $A \in \mathbb{R}_{\max}^{n \times n}$ corresponds to 
the weight of a multi-circuit in the associated graph $G(A)$,
where a multi-circuit is the union of disjoint elementary circuits in the graph.
A multi-circuit is called $\lambda$-maximal if the corresponding term attains the maximum of $\chi_{A}(\lambda)$.
The paper~\cite{Nishida2020} was written under the assumption that for each $\lambda \in \mathbb{R}_{\max}$
 the lengths of any two distinct $\lambda$-maximal multi-circuits are different.
For an algebraic eigenvalue $\lambda$ and a $\lambda$-maximal multi-circuit $\mathcal{C}$,
a nontrivial vector $\bm{x}$ satisfying
\begin{equation*}
(A_{\setminus\mathcal{C}} \oplus \lambda \otimes E_\mathcal{C}) \otimes \bm{x} 
= (A_\mathcal{C} \oplus \lambda \otimes E_{\setminus\mathcal{C}}) \otimes \bm{x}
\end{equation*}
is called an algebraic eigenvector of $A$ with respect to $\lambda$,
where matrices $A_\mathcal{C}, A_{\setminus\mathcal{C}}, E_\mathcal{C}$, and $E_{\setminus\mathcal{C}}$ and defined in Section 2.4.
We have proved in~\cite{Nishida2020} that algebraic eigenvectors possess many good properties 
so that we regard them as the extension of eigenvectors:
(i) all eigenvectors are also algebraic eigenvectors
(ii) the set of all algebraic eigenvectors with respect to an algebraic eigenvalue is a max-plus subspace, called algebraic eigenspace,
and (iii) the dimension of the algebraic eigenspace does not exceed the multiplicity of the algebraic eigenvalue
as a root of the characteristic polynomial.
Compared to a similar approach in the supertropical algebra~\cite{Izhakian2010, Izhakian2011a, Izhakian2011b, Izhakian2011c},
the concept of algebraic eigenvectors has advantages in dealing with bases of algebraic eigenspaces.
\par
In the present paper, we investigate further properties of algebraic eigenvectors.
In the conventional linear algebra, eigenvectors with respect to distinct eigenvalues are linearly independent
and hence the sum of eigenspaces becomes the direct sum.
In the max-plus algebra, independence of vectors can be considered in several ways
because various kinds of ranks of matrices are not equivalent~\cite{Akian2009,Develin2005}.
For example, the tropical rank of a matrix is the maximum order of nonsingular square submatrix.
The column rank is the cardinality of the minimum spanning set of 
the max-plus subspace spanned by the columns of the matrix. 
The tropical rank is less than or equal to the column rank; sometimes the equality does not hold.  
In the sense of tropical ranks, algebraic eigenvectors with respect to distinct algebraic eigenvalues can be dependent.
More precisely, the $n \times r$ matrix consisting of $r$ algebraic eigenvectors with respect to 
$r$ distinct algebraic eigenvalues may have the rank less than $r$.
However, even in that case, it may happen that each of the $r$ algebraic eigenvectors can not be expressed
as a linear combination of other $(r-1)$ vectors.
This means that the column rank of the $n \times r$ matrix consisting of $r$ algebraic eigenvectors is $r$.
In Section 3, we discuss the independence of algebraic eigenvectors in the latter sense.
Note that the independence of supertropical eigenvectors is discussed in the former sense~\cite{Niv2017}.
We first prove that the algebraic eigenspaces with respect to two distinct algebraic eigenvalues
have no intersection except for the zero vector.
Next, we consider the sum space of the algebraic eigenspaces with respect to all algebraic eigenvalues.
The union of the bases of all algebraic eigenspaces of course spans the sum space.
We prove that, under some weak conditions, the union of the bases becomes a basis of the sum space.
\par
For symmetric matrices, an important result in the conventional linear algebra is 
that eigenvectors with respect to distinct eigenvalues are orthogonal to each other.
In the max-plus algebra, the orthogonality of vectors is defined in the sense of tropical geometry~\cite{Develin2004, Grigoriev2018}.
In Section 4, we prove that algebraic eigenvectors of a max-plus symmetric matrix
with respect to distinct algebraic eigenvalues are also orthogonal to each other. 
A max-plus symmetric matrix usually has two distinct $\lambda$-maximal multi-circuits with the same length,
not satisfying the assumption in~\cite{Nishida2020} noted above.
This is because each multi-circuit is accompanied by the reversed one.
Hence, we need to extend the definition of algebraic eigenvectors to apply it to all max-plus matrices.
For this purpose, we consider a perturbation to transform a given matrix to the one satisfying the assumption.
The propositions that guarantee the validity of the modified definition of algebraic eigenvectors 
are presented in Section 2.5 and are proved in Section 5.

\section{The eigenvalue problem over the max-plus algebra}
\label{eigproblem}
\subsection{Max-plus algebra}
Let $\mathbb{R}_{\max} = \mathbb{R}\cup \{\varepsilon \}$ be the set of real numbers
$\mathbb{R}$ with an extra element $\varepsilon := -\infty$.
 We define two operations, addition $\oplus$ and multiplication $\otimes$, on $\mathbb{R}_{\max}$ by   
\begin{align*}
	a \oplus b = \max\{a,b\}, \quad a \otimes b = a+b, \quad a,b \in \mathbb{R}_{\max}.
\end{align*}
Then, $(\mathbb{R}_{\max}, \oplus, \otimes)$ is a commutative semiring
called the max-plus algebra or the tropical semiring.
Here, $\varepsilon$ is the identity element for addition and $e := 0$ is the identity element for multiplication.
For details about the max-plus algebra,
refer to~\cite{Baccelli1992, Butkovic2010, Green1979, Gondran2010, Heidergott2005, Maclagan2015}.
\par
Let $\mathbb{R}_{\max}^n$ and $\mathbb{R}_{\max}^{m \times n }$ be  
the set of  $n$-dimensional max-plus column vectors and the set of $m \times n$ max-plus matrices, respectively.
The operations $\oplus$ and $\otimes$ are extended to max-plus vectors and matrices as in the conventional linear algebra.
Let $\mathcal{E}$ denote the max-plus zero vector or the zero matrix, whose all entries are $\varepsilon$,
and let $E_{n}$ denote the max-plus unit matrix of order $n$.
\par
For a matrix $A = (a_{ij}) \in\mathbb{R}_{\max}^{n \times n}$, we define the determinant of $A$ by 
\begin{align*}
	\det A = \bigoplus_{\pi \in S_n} \bigotimes_{i=1}^n a_{i \pi(i)}, 
\end{align*}
where $S_n$ denotes the symmetric group of order $n$.
A square matrix $A$ is called nonsingular
if the maximum in $\det A$ is attained with precisely one permutation; otherwise, it is called singular. 
The nonsingularity of max-plus matrices is equivalent to the triviality of their kernels in the sense of tropical geometry.
Generally, the tropical kernel of $A = (a_{ij}) \in\mathbb{R}_{\max}^{m \times n}$ is 
the set of vectors $\bm{x} = {}^t\!(x_{1},x_{2},\dots,x_{n})$ such that the maximum 
\begin{align*}
		a_{i1} \otimes x_1 \oplus a_{i2} \otimes x_2 \oplus \cdots \oplus a_{in} \otimes x_n
\end{align*}
is attained with at least two terms for each $i=1,2,\dots,m$.
\begin{proposition}[\cite{Akian2012}] \label{ker}
	A matrix $A = (a_{ij}) \in \mathbb{R}_{\max}^{n \times n}$ is nonsingular if and only if
	the tropical kernel of $A$ is the trivial set $\{ \mathcal{E} \}$.
\end{proposition}
\subsection{Max-plus matrices and graphs}
For a matrix $A = (a_{ij}) \in \mathbb{R}_{\max}^{n \times n}$, 
we define a weighted digraph $G(A) = (V,E,w)$ as follows.
The vertex set and the edge set are $V=\{1,2,\dots,n\}$ and $E=\{ (i,j) \ |\ a_{ij} \neq \varepsilon \}$, respectively,
and the weight function $w: E \to \mathbb{R}$ is defined by $w((i,j)) = a_{ij}$ for $(i,j) \in E$.
For a circuit $C=(i_{0},i_{1},\dots,i_{\ell -1}, i_{\ell} = i_{0})$ in $G(A)$, 
the number $\ell(C) := \ell$ is called the length of $C$ and 
the sum $w(C) := \sum_{k=0}^{\ell-1} a_{i_{k}i_{k+1}}$ is called the weight of $C$.
We define the average weight of $C$ by $\text{ave}(C):=w(C)/\ell(C)$.
If $i_{k} \neq i_{k'}$ for $0 \leq k < k' \leq \ell-1$, the circuit $C$ is called elementary.
A union of disjoint elementary circuits is called a multi-circuit and
its length and weight are defined as the sum of the lengths and weights of the circuits in it, respectively.
\par
For $A \in \mathbb{R}_{\max}^{n \times n}$, we consider the formal matrix power series of the form
\begin{align*}
	A^* := E_{n} \oplus A \oplus A^{\otimes 2}\oplus \cdots.
\end{align*}
If there is no circuit with positive weight in $G(A)$, then it is computed as the finite sum 
\begin{align*}
	A^* = E_{n} \oplus A \oplus A^{\otimes 2}\oplus \cdots \oplus A^{\otimes n-1}.
\end{align*}
\par
\subsection{Eigenvalues and eigenvectors}
A subset $U \subset \mathbb{R}_{\max}^n$ is called a subspace if it is closed with respect to
addition $\oplus$ and scalar multiplication $\otimes$.
A minimal spanning set of a subspace $U$ is called a basis of $U$.
In the max-plus algebra, a basis of a subspace is uniquely determined up to scalar multiplication~\cite{Butkovic2007}.
The number of vectors in a basis is called the dimension of the subspace.
Let $U_{1}, U_{2}, \dots, U_{m}$ be subspaces of $\mathbb{R}_{\max}^{n}$.
We define the sum of these subspaces by
\begin{align*}
	\bigoplus_{k=1}^{m} U_{k} = \left\{ \bigoplus_{k=1}^{m} \bm{x}_{k} \biggm| \bm{x}_{k} \in U_{k} \right\}.
\end{align*}
Note that the symbol $\bigoplus$ signifies the max-plus sum, not the direct sum.
If $\mathcal{B}_{k}$ is a basis of $U_{k}$ for $k=1,2,\dots,m$, 
then $\bigcup_{k=1}^{m} \mathcal{B}_{k}$ obviously spans $\bigoplus_{k=1}^{m} U_{k}$.
\par
For a matrix $A \in \mathbb{R}_{\max}^{n \times n}$, a scalar $\lambda$ is called an eigenvalue of $A$
if there exists a vector $\bm{x} \neq \mathcal{E}$ satisfying
\begin{align*}
	A \otimes \bm{x} = \lambda \otimes \bm{x}.
\end{align*}
Such nontrivial vector $\bm{x}$ is called an eigenvector of $A$ with respect to $\lambda$.
The set of all eigenvectors of $A$ with respect to $\lambda$ together with the zero vector $\mathcal{E}$
is denoted by $U(A,\lambda)$ and called the eigenspace of $A$ with respect to $\lambda$.
It is easily verified that $U(A,\lambda)$ is a subspace of $\mathbb{R}_{\max}^{n}$.
Here, we summarize the results in the literature on the max-plus eigenvalue problem,
 e.g.,~\cite{Baccelli1992,Butkovic2010,Heidergott2005}.
\begin{proposition}\label{eigval1}
	For a matrix $A \in \mathbb{R}_{\max}^{n \times n}$, 
	the maximum value of the average weights of all elementary circuits in $G(A)$ is the maximum eigenvalue of $A$.
\end{proposition}
Let $\lambda(A)$ be the maximum eigenvalue of $A$.
A circuit in $G(A)$ with the average weight $\lambda(A)$ is called critical. 
The subgraph $G^c(A)$ of $G(A)$ consisting of all critical circuits is called the critical graph.
\begin{proposition}\label{eigvec1}
	Let $\bm{g}_k$ be the $k$th column of $((-\lambda(A)) \otimes A)^*$.
	Then, $\bm{g}_k$ is an eigenvector of $A$ with respect to $\lambda$ if and only if $k$ is a vertex in $G^c(A)$.
	Further, the set $\{ \bm{g}_k\ |\ k \in K\}$ is a basis of the eigenspace $U(A,\lambda)$,
	where  $K$ is a set of vertices with exactly one vertex from each connected component of $G^c(A)$.
\end{proposition}
\subsection{Max-plus characteristic polynomials and algebraic eigenvectors}
A (univariate) polynomial in the max-plus algebra has the form
\begin{align*}
	f(t) = c_0 \oplus c_1 \otimes t \oplus c_2 \otimes t^{\otimes 2} \oplus \cdots \oplus c_n \otimes t^{\otimes n},
	\quad c_0,c_1,\dots,c_n \in \mathbb{R}_{\max}, c_{n} \neq \varepsilon.
\end{align*}
Max-plus univariate polynomials are piecewise linear functions on $\mathbb{R}_{\max}$.
Every polynomial can be factorized into a product of linear factors as
\begin{align*}
	f(t) = c_{n} \otimes (t \oplus r_1)^{\otimes p_1} \otimes (t \oplus r_2)^{\otimes p_2} \otimes \cdots 
		\otimes (t \oplus r_m)^{\otimes p_m}.
\end{align*}
Then, $r_i$ and $p_i$ are called a root of $f(t)$ and its multiplicity, respectively.
In the graph of the piecewise linear function $f(t)$, a root is a undifferentiable point of $f(t)$ and its multiplicity is
the difference in the slopes of the lines around the point.
\par
The characteristic polynomial of $A = (a_{ij}) \in \mathbb{R}_{\max}^{n \times n}$ is defined by
\begin{align*}
	\chi_A(t) := \det (A \oplus t \otimes E_{n}).
\end{align*}
As in the conventional algebra, the characteristic polynomial of a matrix is closely related to the eigenvalue problem.
\begin{theorem}[\cite{Green1983}]
For a matrix $A \in \mathbb{R}^{n \times n}_{\max}$, 
the maximum root of the characteristic polynomial is the maximum eigenvalue of $A$.
\end{theorem}
\begin{theorem}[\cite{Akian2006}]
All eigenvalues of a matrix $A \in \mathbb{R}^{n \times n}_{\max}$ are roots of its characteristic polynomial.
\end{theorem}
Roots of $\chi_{A}(t)$ are also called algebraic eigenvalues of $A$~\cite{Akian2006}.
Note that there are $n$ algebraic eigenvalues counting multiplicities.
\par
When we expand the polynomial $\chi_{A}(t)$, the coefficient of $t^{\otimes k}$ is the maximum weight of 
multi-circuits in $G(A)$ with lengths $n-k$.
For $\lambda \in \mathbb{R}$, a multi-circuit $\mathcal{C}$ satisfying 
$\chi_{A}(\lambda) = w(\mathcal{C}) \otimes t^{\otimes (n-\ell(\mathcal{C}))}$ is called $\lambda$-maximal.
An $\varepsilon$-maximal multi-circuit is defined as the $\mu$-maximal multi-circuit for sufficiently small real number $\mu$.
If $\lambda$ is a finite algebraic eigenvalue, then there exist at least two $\lambda$-maximal multi-circuits
with the different lengths.
\par
In the recent paper~\cite{Nishida2020},
the authors have introduced algebraic eigenvectors with respect to algebraic eigenvalues 
under the following assumption for a matrix $A$:\\[10pt]
($\bigstar$)\qquad
	\begin{minipage}{0.8\textwidth}
	for each $\lambda \in \mathbb{R}_{\max}$,
	the lengths of any two distinct $\lambda$-maximal multi-circuits in $G(A)$ are different. 
\end{minipage}\\[5pt]
First, for $A \in \mathbb{R}^{n \times n}_{\max}$ and a multi-circuit $\mathcal{C}$,
we define four types of matrices 
$A_\mathcal{C}, A_{\setminus\mathcal{C}}, E_\mathcal{C}$ and $E_{\setminus\mathcal{C}}$ as follows:
\begin{align*}
	&[A_\mathcal{C}]_{ij} = 
	\begin{cases} a_{ij} & \text{if $(i,j) \in E(\mathcal{C})$,} \\ \varepsilon & \text{otherwise,} \end{cases}
	\quad [A_{\setminus\mathcal{C}}]_{ij} = 
	\begin{cases} \varepsilon & \text{if $(i,j) \in E(\mathcal{C})$,} \\ a_{ij} & \text{otherwise,} \end{cases}\\
	&[E_\mathcal{C}]_{ij} = 
	\begin{cases} e & \text{if $i =j$, $i \in V(\mathcal{C})$,} \\ \varepsilon & \text{otherwise,} \end{cases}
	\quad[E_{\setminus\mathcal{C}}]_{ij} = 
	\begin{cases} e & \text{if $i =j$, $i \not\in V(\mathcal{C})$,} \\ \varepsilon & \text{otherwise.} \end{cases}
\end{align*}
Here, $[*]_{ij}$ denotes the $(i,j)$ entry of the matrix and 
 $V(\mathcal{C})$ and $E(\mathcal{C})$ denote the vertex set and the edge set of $\mathcal{C}$, respectively.
\begin{proposition}[\cite{Nishida2020}]
	Suppose $A \in \mathbb{R}^{n \times n}_{\max}$ satisfies {\rm ($\bigstar$)}.
	Then, $\lambda \in \mathbb{R}_{\max}$ is an algebraic eigenvalue of $A$ if and only if
	there exist a $\lambda$-maximal multi-circuit $\mathcal{C}$ and a vector $\bm{x} \neq \mathcal{E}$ such that
	\begin{align}
		(A_{\setminus\mathcal{C}} \oplus \lambda \otimes E_\mathcal{C}) \otimes \bm{x} 
		= (A_\mathcal{C} \oplus \lambda \otimes E_{\setminus\mathcal{C}}) \otimes \bm{x}. \label{algeigveceq}
	\end{align}
	This nontrivial vector $\bm{x}$ is called an algebraic eigenvector of $A$ with respect to $\lambda$.
\end{proposition}
\begin{proposition}[\cite{Nishida2020}] \label{balc}
	Suppose $A \in \mathbb{R}^{n \times n}_{\max}$ satisfies {\rm ($\bigstar$)}.
	A vector $\bm{x} \neq \mathcal{E}$ is an algebraic eigenvector of $A$ with respect to $\lambda \in \mathbb{R}$ if and only if
	it is an eigenvector of 
	\begin{align*}
		B_{A,\lambda,\mathcal{C}} := (A_\mathcal{C} \oplus \lambda \otimes E_{\setminus\mathcal{C}})^{-1}
			\otimes (A_{\setminus\mathcal{C}} \oplus \lambda \otimes E_\mathcal{C})
	\end{align*}
	with respect to $0$.
\end{proposition}
Algebraic eigenvectors are also characterized in terms of adjugate matrices.
The adjugate matrix $\mathrm{adj}(A) \in \mathbb{R}_{\max}^{n \times n}$ is defined by $[\mathrm{adj}(A)]_{ij} = \det A^{(j,i)}$,
where $A^{(j,i)} \in \mathbb{R}_{\max}^{(n-1) \times (n-1)}$ is obtained from $A$ by deleting the $j$th row and the $i$th column.
Let $\Gamma(A,\lambda) := \mathrm{adj}(A \oplus \lambda \otimes E_{n})$.
\begin{proposition} \label{balc-adj}
	Let $\lambda \in \mathbb{R}$ be an algebraic eigenvalue of $A \in \mathbb{R}_{\max}^{n \times n}$
	and $\mathcal{C}$ be a $\lambda$-maximal multi-circuit.
	Then, $(B_{A,\lambda,\mathcal{C}})^{*}$ is equivalent to $\Gamma(A,\lambda)$
	up to permutation and scaling of columns.
	In particular, the $i$th column of $(B_{A,\lambda,\mathcal{C}})^{*}$
	is a scalar multiple of the $\sigma(i)$th column of $\Gamma(A,\lambda)$, 
	where $\sigma(i)$ is the succeeding vertex of $i$ in $\mathcal{C}$ if $i \in V(\mathcal{C})$
	and otherwise $\sigma(i) = i$.
\end{proposition}
To prove this proposition, we use the following result in~\cite{Yoeli1961}.
\begin{lemma}[\cite{Yoeli1961}] \label{yoeli}
	 If all diagonal entries of a square matrix $P$ are $0$ and $G(P)$ has no circuit with the positive weight,
	then $\mathrm{adj}(P) = P^{*}$.
\end{lemma}
\begin{lemma} \label{adj1}
	Let $P = (p_{ij}) \in \mathbb{R}_{\max}^{n \times n}$ be a generalized permutation matrix, i.e.,
	there is a permutation $\sigma \in S_{n}$ such that $p_{ij} \neq \varepsilon$ if and only if $j = \sigma(i)$.
	Then, for any $Q \in \mathbb{R}_{\max}^{n \times n}$, $\mathrm{adj}(Q)$ is equivalent to $\mathrm{adj}(P \otimes Q)$
	up to permutation and scaling of columns.
\end{lemma}
\begin{proof}
Let $Q^{(i,j)}$ be the matrix obtained by deleting the $i$th row and the $j$th column of $Q$.
Since the $k$th row of $P \otimes Q$ is identical to the $\sigma(k)$th row of $Q$ multiplied by $p_{k\sigma(k)}$,
we have
\begin{align*}
	\det (P \otimes Q)^{(i,j)} = \det Q^{(\sigma(i),j)} \otimes \bigotimes_{k \neq i} p_{k\sigma(k)}.
\end{align*}
Hence, the $i$th column of $\mathrm{adj}(P \otimes Q)$ is identical to 
the $\sigma(i)$th column of $\mathrm{adj}(Q)$ multiplied by $\bigotimes_{k \neq i} p_{k\sigma(k)}$, proving the lemma.
\end{proof}
\begin{proof}[Proof of Proposition \ref{balc-adj}]
	If $\mathcal{C}$ is $\lambda$-maximal, then $G(B_{A,\lambda,\mathcal{C}})$ does not contain any circuits 
	with the positive weights.
	By Lemma \ref{yoeli}, we have
	\begin{align*}
		(B_{A,\lambda,\mathcal{C}})^{*} = (B_{A,\lambda,\mathcal{C}} \oplus E_{n})^{*}
		= \mathrm{adj}(B_{A,\lambda,\mathcal{C}} \oplus E_{n}).
	\end{align*}
	On the other hand, we have
	\begin{align*}
		B_{A,\lambda,\mathcal{C}} \oplus E_{n} 
		&= (A_\mathcal{C} \oplus \lambda \otimes E_{\setminus\mathcal{C}})^{-1} \otimes
			( (A_{\setminus\mathcal{C}} \oplus \lambda \otimes E_\mathcal{C}) \oplus
			(A_\mathcal{C} \oplus \lambda \otimes E_{\setminus\mathcal{C}}) ) \\
		&= (A_\mathcal{C} \oplus \lambda \otimes E_{\setminus\mathcal{C}})^{-1} \otimes
			(A \oplus \lambda \otimes E_{n}).
	\end{align*}
	By Lemma \ref{adj1}, we have the conclusion.
\end{proof}
For a (not necessarily $\lambda$-maximal) multi-circuit $\mathcal{C}$,
let $W(A,\lambda,\mathcal{C})$ be the set of vectors satisfying \eqref{algeigveceq}.
\begin{proposition}[\cite{Nishida2020}] 
	For $\lambda \in \mathbb{R}$, a $\lambda$-maximal multi-circuit $\mathcal{C}$ and any multi-circuit $\mathcal{C}'$,
	we have $W(A,\lambda,\mathcal{C}') \subset W(A,\lambda,\mathcal{C})$.
	In particular, if $\mathcal{C}'$ is also $\lambda$-maximal, 
	then $W(A,\lambda,\mathcal{C}) = W(A,\lambda,\mathcal{C}')$.
\end{proposition} 
Hence, the subspace $W(A,\lambda,\mathcal{C})$ does not depend on the choice of a $\lambda$-maximal multi-circuits $\mathcal{C}$.
We call $W(A,\lambda,\mathcal{C})$ the algebraic eigenspace of $A$ with respect to $\lambda$.
\begin{proposition}[\cite{Nishida2020}] \label{dimsub}
	Suppose $A \in \mathbb{R}^{n \times n}_{\max}$ satisfies {\rm ($\bigstar$)}.
	The dimension of the algebraic eigenspace of $A$ with respect to $\lambda$ does not exceed
	the multiplicity of $\lambda$ in the characteristic polynomial of $A$.
\end{proposition}
\begin{remark}
	Even when $A \in \mathbb{R}^{n \times n}_{\max}$ does not satisfy {\rm ($\bigstar$)},
	there is a vector $\bm{x} \neq \mathcal{E}$ satisfying \eqref{algeigveceq} 
	if $\lambda \in \mathbb{R}_{\max}$ is an algebraic eigenvalue of $A$.
	However, Proposition \ref{dimsub} may not hold for that case.
	In particular, there may be a nontrivial vector satisfying \eqref{algeigveceq} 
	if $\lambda$ is not an algebraic eigenvalue of $A$.
\end{remark}
\subsection{Extension of the definition of algebraic eigenvectors}
Although generic matrices satisfy the assumption ($\bigstar$),
max-plus symmetric matrices, which we focus on in Section 4, does not satisfy it 
because every circuit with length larger than two is accompanied by the reversed one.
Thus, we need to extend the definition of algebraic eigenvectors so that it can be applied to all square matrices.
For $A = (a_{ij}) \in \mathbb{R}^{n \times n}_{\max}$ and $\bm{\zeta} = (\zeta_{ij}) \in [0,1]^{n \times n}$,
we define the matrix $A(\bm{\zeta};\delta)$ with positive real parameter $\delta$ by
\begin{align*}
	[A(\bm{\zeta};\delta)]_{ij} = a_{ij} - \zeta_{ij} \delta
\end{align*}
for all $i,j$.
Let $Z$ be the set of matrices $\bm{\zeta} \in [0,1]^{n \times n}$ with the following property:
there exists $\delta' > 0$ such that $A(\bm{\zeta};\delta)$ satisfies the assumption ($\bigstar$) 
for all $\delta$ with $0 < \delta < \delta'$.
As $\mathbb{R}$ is an infinite field,  $Z \neq \emptyset$.
When $\delta > 0$ is sufficiently small, 
if $\lambda$ is a root of $\chi_{A}(t)$ with the multiplicity $m$ 
then the interval $\mathcal{B}(\lambda,n\delta) := [\lambda - n\delta, \lambda + n\delta]$ contains 
exactly $m$ roots of  $\chi_{A(\bm{\zeta};\delta)}(t)$ counting multiplicities.
We set $\mathcal{B}(\varepsilon,n\delta) := \{ \varepsilon \}$.
If $\chi_{A}(t)$ has the root $\lambda = \varepsilon$ with the multiplicity $m$, so does $\chi_{A(\bm{\zeta};\delta)}(t)$.
Now, we define the set
\begin{align}
	W(A,\lambda) = \bigcap_{\bm{\zeta} \in Z} \lim_{\delta \to +0} 
		\bigoplus_{\lambda' \in \mathcal{B}(\lambda,n\delta)} 
		W(A(\bm{\zeta};\delta), \lambda', \mathcal{C}_{\lambda'}(\bm{\zeta};\delta)), \label{modalgeigsp}
\end{align}
where $\mathcal{C}_{\lambda'}(\bm{\zeta};\delta)$ is a $\lambda'$-maximal multi-circuit in $G(A(\bm{\zeta};\delta))$.
The limit of parametrized subsets $\{S(\delta)\}_{\delta > 0}$ of $\mathbb{R}_{\max}^{n}$ is defined by
\begin{align*}
	\lim_{\delta \to +0} S(\delta) = \{ \bm{x} \in \mathbb{R}_{\max}^{n} \ |\ 
		{}^{\exists}\! \{\bm{x}(\delta)\}_{\delta > 0}, \bm{x}(\delta) \in S(\delta) 
		\text{ such that } \lim_{\delta \to +0} \bm{x}(\delta) = \bm{x} \}.
\end{align*}
For $\bm{\zeta} \in Z$ and sufficiently small $\delta > 0$,
it follows from the definition of $Z$ that
$W(A(\bm{\zeta};\delta), \lambda', \mathcal{C}_{\lambda'}(\bm{\zeta};\delta)) \neq \{\mathcal{E}\}$ 
if and only if $\lambda'$ is a root of $\chi_{A(\bm{\zeta};\delta)}(t)$.
Hence, the infinite sum of max-plus subspaces in~\eqref{modalgeigsp} is in fact finite and can be written as
\begin{align*}
	W(A,\lambda) = \bigcap_{\bm{\zeta} \in Z} \lim_{\delta \to +0} 
		\bigoplus_{p} W(A(\bm{\zeta};\delta), \lambda^{(p)}(\bm{\zeta};\delta), \mathcal{C}^{(p)}(\bm{\zeta};\delta)),
\end{align*}
where $\lambda^{(p)}(\bm{\zeta};\delta), p=1,2, \dots, $ are the (finite number of) roots of $\chi_{A(\bm{\zeta};\delta)}(t)$
in $\mathcal{B}(\varepsilon,n\delta)$ and 
$\mathcal{C}^{(p)}(\bm{\zeta};\delta)$ denotes $\mathcal{C}_{\lambda^{(p)}(\bm{\zeta};\delta)}(\bm{\zeta};\delta)$.
The following propositions show that $W(A,\lambda)$ deserves to be called the algebraic eigenspace of $A$ with respect to $\lambda$.
\begin{proposition} \label{subspace}
	$W(A,\lambda)$ is a max-plus subspace of $\mathbb{R}_{\max}^{n}$.
\end{proposition}
\begin{proposition} \label{ifonlyif}
	$\lambda \in \mathbb{R}_{\max}$ is an algebraic eigenvalue of $A$ if and only if $W(A,\lambda) \neq \{\mathcal{E} \}$.
\end{proposition}
\begin{proposition} \label{dimsub2}
	The dimension of $W(A,\lambda)$ does not exceed the multiplicity of $\lambda$ in the characteristic polynomial of $A$.
\end{proposition}
\begin{proposition} \label{generalize}
	For any $\lambda$-maximal multi-circuit $\mathcal{C}$, we have $W(A,\lambda) \subset W(A,\lambda,\mathcal{C})$.
	Furthermore, if $A$ satisfies the assumption {\rm ($\bigstar$)}, then $W(A,\lambda) = W(A,\lambda,\mathcal{C})$.
\end{proposition}
The proofs of these propositions are presented in Section 5.
According to Proposition \ref{generalize}, we call $W(A,\lambda)$ the algebraic eigenspace of $A$ with respect to $\lambda$
without fear of confusion.
Nontrivial vectors in $W(A,\lambda)$ are called algebraic eigenvectors of $A$ with respect to $\lambda$.
\par
By Proposition \ref{generalize}, the algebraic eigenspace $W(A,\lambda)$ for $\lambda \neq \varepsilon$ 
is spanned by certain columns of $\Gamma(A,\lambda) := \mathrm{adj}(A \oplus \lambda \otimes E_{n})$.
To find the spanning set of $W(A,\lambda)$, we investigate the relation between
$\lambda$-maximal multi-circuits in $G(A)$ and critical circuits in $G(B_{A,\lambda,\mathcal{C}})$.
For a multi-circuit $\mathcal{C}$, the corresponding permutation $\sigma_{\mathcal{C}}$ is defined as follows:
$\sigma_{\mathcal{C}}(i)$ is the succeeding vertex of $i$ in $\mathcal{C}$ if $i \in V(\mathcal{C})$
and otherwise $\sigma_{\mathcal{C}}(i) = i$.
Let $\overleftarrow{\mathcal{C}}$ denote the multi-circuit obtained by reversing all edges in a multi-circuit $\mathcal{C}$.
The following procedure is introduced in~\cite{Nishida2020}.
\begin{definition} \label{mapvarphi}
	Let $\lambda \in \mathbb{R}$ be an algebraic eigenvalue of $A \in \mathbb{R}_{\max}^{n \times n}$
	and $\mathcal{C}$ be a $\lambda$-maximal multi-circuit.
	Further, let $\mathcal{D}$ be a multi-circuits consisting of some critical circuits in $G(B_{A,\lambda,\mathcal{C}})$.
	We define a $\lambda$-maximal multi-circuit $\mathcal{C}'$ in $G(A)$ by the procedure below 
	and denote it by $\varphi_{\mathcal{C}}(\mathcal{D})$.
	\begin{enumerate}
		\item 
		Set $\mathcal{C}' := \emptyset$.
		\item 
		Choose any edge of $\mathcal{D}$ that is not in $E(\overleftarrow{\mathcal{C}})$
		and denote the terminal vertex of that edge by $i$.
		We define the initial sequence of vertices by $\hat{C} := (i)$.
		\item 
		The succeeding vertex of $i$ in $\hat{C}$ is determined by the following rules.
		\begin{enumerate}
			\item
			 If $i \not\in V(\mathcal{C})$, let $i'$ be the succeeding vertex of $i$ in $\mathcal{D}$.
			Append $i'$ to $\hat{C}$ and set $i:=i'$.
			\item 
			If $i \in V(\mathcal{C})$ and $\sigma_{\mathcal{C}}(i) \not\in V(\mathcal{D})$, 
			append $\sigma_{\mathcal{C}}(i)$ to $\hat{C}$ and set $i:=\sigma_{\mathcal{C}}(i)$.
			\item 
			If $i \in V(\mathcal{C})$ and $\sigma_{\mathcal{C}}(i) \in V(\mathcal{D})$, 
			let $i'$ be the succeeding vertex of $\sigma_{\mathcal{C}}(i)$ in $\mathcal{D}$.
			Append $i'$ to $\hat{C}$ and set $i:=i'$.
		\end{enumerate}
		\item 
		Repeat 3 until the original vertex $i$ selected in 2 appears again.
		If we return to $i$, append the circuit $\hat{C}$ to $\mathcal{C}'$.
		\item
		 Repeat 2--4 while there exist edges (or corresponding terminal vertices) satisfying 2.
		\item 
		Append all circuits in $\mathcal{C}$ that have no common vertices with $\mathcal{D}$ to $\mathcal{C}'$.
		\item 
		Find all loops on $V(\mathcal{D}) \setminus V(\mathcal{C}')$ whose weights are greater than $\lambda$.
		Append them to $\mathcal{C}'$.
	\end{enumerate}
\end{definition}
The inverse of this procedure is defined as follows.
\begin{definition} \label{mappsi}
	Let $\lambda \in \mathbb{R}$ be an algebraic eigenvalue of $A \in \mathbb{R}_{\max}^{n \times n}$
	and $\mathcal{C, \mathcal{C}'}$ be $\lambda$-maximal multi-circuits in $G(A)$.
	We define a multi-circuit $\mathcal{D}$ consisting of critical circuits in $G(B_{A,\lambda,\mathcal{C}})$
	by the procedure below and denote it by $\psi_{\mathcal{C}}(\mathcal{C}')$.
	\begin{enumerate}
		\item 
		Set $\mathcal{D} := \emptyset$.
		\item 
		Choose any vertex $i \in V(\mathcal{C}') \setminus V(\mathcal{C})$.
		We define the initial sequence of vertices by $\hat{D} := (i)$.
		\item 
		The succeeding vertex of $i$ in $\hat{D}$ is determined by the following rules.
		\begin{enumerate}
			\item
			If $i \in V(\mathcal{C}') \setminus V(\mathcal{C})$, then append $\sigma_{\mathcal{C}'}(i)$ to $\hat{D}$ 
			and set $i := \sigma_{\mathcal{C}'}(i)$.
			\item 
			If $i \in V(\mathcal{C})$, let $i'$ be the preceding vertex of $i \in \mathcal{C}$. 
			If $i' \in V(\mathcal{C}')$, then append $\sigma_{\mathcal{C}'}(i')$ to $\hat{D}$ and set $i := \sigma_{\mathcal{C}'}(i')$;
			otherwise append $i'$ to $\hat{D}$ and set $i := i'$
		\end{enumerate}
		\item 
		Repeat 3 until the original vertex $i$ selected in 2 appears again.
		If we return to $i$, append the circuit $\hat{D}$ to $\mathcal{D}$.
		\item
		 Repeat 2--4 while there remains a vertex in $V(\mathcal{C}') \setminus V(\mathcal{C})$.
	\end{enumerate}
\end{definition}
Now, we describe the way to detect which column of $\Gamma(A,\lambda)$
 is an algebraic eigenvector of $A$ with respect to $\lambda$.
\begin{definition} \label{validclass}
	Let $A \in \mathbb{R}_{\max}^{n \times n}$ and $\lambda \in \mathbb{R}$.
	We consider the following equivalence relation on the set $\{1,2,\dots,n\}$:
	if the $i$th column of $\Gamma(A,\lambda)$ is a scalar multiple of the $j$th column of it, then $i$ and $j$ are equivalent.
	An equivalent class $H$ is called valid if there exists two distinct $\lambda$-maximal multi-circuits 
	$\mathcal{C}$ and $\mathcal{C}'$ with $\ell(\mathcal{C}) < \ell(\mathcal{C}')$ 
	such that $V(\psi_{\mathcal{C}}(\mathcal{C}')) \subset \sigma_{\mathcal{C}}(H)$ 	in $G(B_{A,\lambda,\mathcal{C}})$.
\end{definition}
\begin{proposition} \label{modifeigspace}
	For each algebraic eigenvalue $\lambda \in \mathbb{R}$ of $A \in \mathbb{R}_{\max}^{n \times n}$, 
	$W(A, \lambda)$ is spanned by the columns of $\Gamma(A,\lambda)$ corresponding to indices in valid equivalent classes.
\end{proposition}
\proof
	See Section 5.
\endproof
\begin{example} \label{algeigvecexm}
	Consider a matrix
	\begin{align*}
		A = \begin{pmatrix} 
			6 & 5 & 0 \\ 5 & 4 & \varepsilon \\ 0 & \varepsilon & 2
		\end{pmatrix}.
	\end{align*}
The algebraic eigenvalues of $A$ are $6, 4$ and $2$.
This matrix does not satisfy ($\bigstar$) because $G(A)$ contains two $\lambda$-maximal multi-circuits
$\{(1,1), (2,2), (3,3) \}$ and $\{(1,2,1), (3,3)\}$ with lengths $3$ for any $\lambda \leq 2$.
Take an algebraic eigenvalue $\lambda = 2$. There are four $\lambda$-maximal multi-circuits
$\{(1,1), (2,2) \}, \{(1,2,1)\}$, $\{(1,1), (2,2), (3,3)\}$ and $\{(1,2,1),(3,3)\}$.
Let $\mathcal{C} = \{(1,1), (2,2), (3,3)\}$.
We have
\begin{align*}
	(B_{A,\lambda,\mathcal{C}})^{*} = \begin{pmatrix} 
		0 & -1 & -6 \\ 1 & 0 & -5 \\ -2 & -3 & 0
	\end{pmatrix}
	= \Gamma(A,\lambda) \otimes \begin{pmatrix}
		-6 & \varepsilon & \varepsilon \\ \varepsilon & -5 & \varepsilon \\ \varepsilon & \varepsilon & -4
	\end{pmatrix}.
\end{align*}
All columns of $(B_{A,\lambda,\mathcal{C}})^{*}$ satisfy the equation \eqref{algeigveceq} for $\lambda = 2$.
However, only the third column is in a valid class because $\psi_{\mathcal{C}_{1}}(\mathcal{C}_{2}) = \{(3,3)\}$
if $\mathcal{C}_{1} = \{(1,1), (2,2) \}$ or $\{(1,2,1)\}$ and $\mathcal{C}_{2} = \{(1,1), (2,2), (3,3)\}$ or $\{(1,2,1),(3,3)\}$.
Hence, the algebraic eigenspace with respect to $\lambda = 2$ is
\begin{align*}
	W(A,\lambda) = \{ \alpha \otimes {}^t\!(-6,-5,0) \ |\ \alpha \in \mathbb{R}_{\max} \}.
\end{align*}
Note that ${}^t\!(0,1,-2)$, the first column of $(B_{A,\lambda,\mathcal{C}})^{*}$,
 satisfies the equation \eqref{algeigveceq} for any $\lambda \leq 2$,
which implies this vector is independent from the choice of $\lambda$ and hence does not suitable for an algebraic eigenvector.
\end{example}
\section{Independence of algebraic eigenvectors}
One important feature of eigenvectors in the conventional linear algebra is
that eigenvectors with respect to distinct eigenvalues are independent.
In this section, we discuss the max-plus analogue of this fact.
First, we see an example taken from~\cite{Niv2017}, which is written in the supertropical settings.
\par
\begin{example}
	Consider a matrix
	\begin{align*}
		A = \begin{pmatrix}
		10 & 10 & 9 & \varepsilon \\
		9 & 1 & \varepsilon & \varepsilon \\
		\varepsilon & \varepsilon & \varepsilon & 9 \\
		9 & \varepsilon & \varepsilon & \varepsilon 
		\end{pmatrix}.
	\end{align*}
	The algebraic eigenvalues of $A$ are $10, 9, 8, 1$ and the corresponding algebraic eigenvectors are
	\begin{align*}
		\bm{x}_{1} = {}^t\!(2,1,0,1), \quad \bm{x}_{2} = {}^t\!(0,0,0,0), \quad
		 \bm{x}_{3} = {}^t\!(0,1,2,1), \quad \bm{x}_{4} =  {}^t\!(0,15,16,8),
	\end{align*}
	respectively. 
	Let $P = \begin{pmatrix} \bm{x}_{1} & \bm{x}_{2} & \bm{x}_{3} & \bm{x}_{4} \end{pmatrix} \in \mathbb{R}_{\max}^{4 \times 4}$.
	Then, $\det P = 19$ is attained with three permutations: $(2\ 4), (3\ 4)$ and $(2\ 3\ 4)$.
	Hence, $P$ is singular.		
\end{example}
This example shows that algebraic eigenvectors $\bm{x}_{1}, \bm{x}_{2}, \bm{x}_{3}$ and $\bm{x}_{4}$ are dependent
in the sense of the tropical rank of the matrix~\cite{Develin2005}.
On the other hand, it is easily checked that none of the four vectors can be expressed
as the max-plus linear combination of the other three.
In this sense, $\bm{x}_{1}, \bm{x}_{2}, \bm{x}_{3}$ and $\bm{x}_{4}$ are independent.
In this section, we adopt the latter definition and show that
every basis vector of an algebraic eigenspace cannot be expressed as a linear combination of
other algebraic eigenvectors.
We first prove the following proposition.
\begin{proposition}
	Let $\lambda_{1}$ and $\lambda_{2}$ be distinct algebraic eigenvalues of $A = (a_{ij}) \in \mathbb{R}_{\max}^{n \times n}$.
	Then, we have $W(A,\lambda_{1}) \cap W(A,\lambda_{2}) = \{\mathcal{E}\}$.
\end{proposition}
\begin{proof}
	Take a vector $\bm{x} = {}^t\!(x_{1},x_{2},\dots,x_{n}) \in W(A,\lambda_{1}) \cap W(A,\lambda_{2})$.
	Then, we have
	\begin{align*}
		(A \oplus \lambda_{1} \otimes E_{n}) \otimes \bm{x} = (A \oplus \lambda_{2} \otimes E_{n}) \otimes \bm{x}
		= A \otimes \bm{x}.
	\end{align*}
	We first consider the case $A \otimes \bm{x} \in \mathbb{R}$.
	Let $\mathcal{C}_{1}$ and $\mathcal{C}_{2}$ be $\lambda_{1}$-maximal and $\lambda_{2}$-maximal multi-circuits
	in $G(A)$, respectively.
	In this case, $\lambda_{1},\lambda_{2} \neq \varepsilon$ and we may assume that $\ell(\mathcal{C}_{1}), \ell(\mathcal{C}_{2}) < n$.
	Then, taking $i_{1} \not\in V(\mathcal{C}_{1})$ and $i_{2} \not\in V(\mathcal{C}_{2})$, we have
	\begin{align*}
		\lambda_{1} \otimes x_{i_{1}}  
			&= [(A_{\mathcal{C}_{1}} \oplus \lambda_{1} \otimes E_{\setminus\mathcal{C}_{1}}) \otimes \bm{x}]_{i_{1}} 
			= [A \otimes \bm{x}]_{i_{1}} \geq \lambda_{2} \otimes x_{i_{1}}, \\
		\lambda_{2} \otimes x_{i_{2}} 
			&= [(A_{\mathcal{C}_{2}} \oplus \lambda_{2} \otimes E_{\setminus\mathcal{C}_{2}}) \otimes \bm{x}]_{i_{2}}
			= [A \otimes \bm{x}]_{i_{2}} \geq \lambda_{1} \otimes x_{i_{2}}.
	\end{align*}
	Here, $[*]_{i}$ denotes the $i$th entry of the vector.
	Hence, we have $\lambda_{1} = \lambda_{2}$, which is a contradiction.
	\par
	If $[A \otimes \bm{x}]_{k} = \varepsilon$ for some $k$,
	let $K = \{ k \ |\ x_{k} = \varepsilon \}$ and $L = \{ k \ |\ x_{k} \neq \varepsilon \}$.
	Then, $a_{ij} = \varepsilon$ for all $(i,j) \in K \times L$.
	If we assume $\lambda_{1} > \lambda_{2} \geq \varepsilon$,
	$[A \otimes \bm{x}]_{k} \geq \lambda_{1} \otimes x_{k}$ implies $L \subset \{ k \ |\ [A \otimes \bm{x}]_{k} \neq \varepsilon \}$.
	Hence, restricting the calculation to rows and columns indexed by $L$, 
	we have $\lambda_{1} = \lambda_{2}$ similarly to the above case, which is a contradiction.
	Thus, we conclude that $\bm{x} = \mathcal{E}$.
\end{proof}
Recall that if $\mathcal{B}_{k}$ is a basis of a subspace $U_{k}$ for $k=1,2,\dots,m$, 
then $\bigcup_{k=1}^{m} \mathcal{B}_{k}$ spans $\bigoplus_{k=1}^{m} U_{k}$.
Now, we present the first main result of the paper.
\begin{theorem} \label{main1}
	Let $\Lambda$ be the set of all algebraic eigenvalues of $A \in \mathbb{R}_{\max}^{n \times n}$ and
	$\mathcal{B}_{\lambda}$ be a basis of the algebraic eigenspace with respect to $\lambda \in \Lambda$.
	If any two critical circuits in $G(B_{A,\lambda,\mathcal{C}})$ are disjoint for all $\lambda \in \mathbb{R}$ and
	$\lambda$-maximal multi-circuits $\mathcal{C}$,
	then $\bigcup_{\lambda \in \Lambda} \mathcal{B}_{\lambda}$ is a basis
	of the sum of all algebraic eigenspaces.
\end{theorem}
Such an assumption described in terms of disjoint circuits
 also appears when we consider a max-plus analogue of Jordan canonical forms of matrices~\cite{Nishida2019}.
Before the proof of the theorem, we present technical lemmas.
\begin{lemma} \label{minimalelem}
	Let $P$ be a matrix with the maximum eigenvalue $0$ and $\bm{p}_{i}$ be the $i$th column of $P^{*}$.
	If $i$ is on the critical circuit in $G(P)$, then $\bm{p}_{i}$ is the minimal element in
	\begin{align*}
		U^{(i)} = \{ (-v_{i}) \otimes \bm{v} \ |\ \bm{v} \in U(P,0), v_{i} \neq \varepsilon  \},
	\end{align*}
	where $U(P,0)$ is the eigenspace of $P$ with respect to $0$.
\end{lemma}
\begin{proof}
	Suppose that there is the $j$th column $\bm{p}_{j}$ of $P^{*}$ such that $(-[\bm{p}_{j}]_{i}) \otimes \bm{p}_{j} \leq \bm{p}_{i}$.
	The $j$th entry of this inequality is $-[\bm{p}_{j}]_{i} \leq [\bm{p}_{i}]_{j}$, yielding $[\bm{p}_{i}]_{j} + [\bm{p}_{j}]_{i} \geq 0$.
	Since the maximum average weight of circuits in $G(P)$ is $0$, the vertices $i$ and $j$ are on the same critical circuit.
	This implies $\bm{p}_{i} = (-[\bm{p}_{j}]_{i}) \otimes \bm{p}_{j}$.
\end{proof}
\begin{lemma} \label{adjproduct}
	The product of the $i$th row of $P = (p_{ij}) \in \mathbb{R}_{\max}^{n \times n}$ and the $i$th column of $\mathrm{adj}(P)$
	is equal to $\det P$ for all $i$.
	In particular, if the maximum in this product is attained with $p_{ij} \otimes [\mathrm{adj}(P)]_{ji}$,
	then there exists a permutation $\sigma \in S_{n}$ attaining $\det P$ such that $\sigma(i) = j$.
\end{lemma}
\begin{proof}
Let $S(i,j)$ denote the set of all bijections $\tau: \{1,2,\dots,n\} \setminus \{ i \} \to \{1,2,\dots,n\} \setminus \{ j \}$.
Then, we have
\begin{align*}
	[\mathrm{adj}(P)]_{ji} = \bigoplus_{\tau \in S(i,j)} \bigotimes_{k \neq i}p_{k \tau(k)}
\end{align*}
and hence 
\begin{align*}
	p_{ij} \otimes [\mathrm{adj}(P)]_{ji} 
	= \bigoplus_{\sigma \in S_{n}, \sigma(i) = j} \bigotimes_{k =1}^{n} p_{k \sigma(k)} \leq \det P.
\end{align*}
Considering the case where the equality holds, we get the latter statement of the lemma.
Taking the maximum for $j = 1,2,\dots,n$, we have
\begin{align*}
	\bigoplus_{j=1}^{n} p_{ji} \otimes [\mathrm{adj}(P)]_{ij} 
	= \bigoplus_{\sigma \in S_{n}} \bigotimes_{k =1}^{n} p_{k \sigma(k)} = \det P.
\end{align*} 
\end{proof}

\begin{proof}[Proof of Theorem \ref{main1}]
	It is sufficient to show that any vector in
	$\bigcup_{\lambda \in \Lambda} \mathcal{B}_{\lambda}$  cannot be expressed by others.
	Take $\lambda \in \Lambda$, $\bm{x} = {}^t\!(x_{1},x_{2},\dots,x_{n}) \in \mathcal{B}_{\lambda}$ and 
	$\lambda$-maximal multi-circuit $\mathcal{C}$ with the minimum length.
	By Proposition \ref{balc} and Proposition \ref{eigvec1}, there is an index $i$ such that 
	$\bm{x}$ is a scalar multiple of the $i$th column of $(B_{A,\lambda,\mathcal{C}})^{*}$.
	(If $A$ does not satisfy ($\bigstar$), we refer to Proposition \ref{modifeigspace}.)
	In particular, we may assume that $i \not\in V(\mathcal{C})$ and 
	there is a circuit $D$ with weight $0$ in $G(B_{A,\lambda,\mathcal{C}})$
	such that $i \in V(D)$ and the length of $\mathcal{C}':= \varphi_{\mathcal{C}}(D)$ is larger than $\ell(\mathcal{C})$, 
	see Proposition \ref{mapvarphi}.
	Further, letting $j$ and $k$ be the preceding and the succeeding vertices of $i$ in $\mathcal{C}'$ respectively,
	we may assume that we have $j \not\in V(\mathcal{C})$ or $(j,k) \in E(\mathcal{C})$.
	Indeed, if there exist two consecutive vertices in $V(\mathcal{C}') \setminus V(\mathcal{C})$, 
	we can take them as $j$ and $i$.
	Otherwise, consider tuples the consecutive vertices $(j,i,k)$ in $\mathcal{C}'$ where $j,k \in V(\mathcal{C})$ and
	$i \not\in V(\mathcal{C})$.
	Since $\mathcal{C}'$ is obtained by replacing the subpath $\mathcal{P}_{jk}$ in $\mathcal{C}$ from $j$ to $k$ 
	with the path $j \to i \to k$ for all such tuples $(j,i,k)$, 
	one of the subpaths $\mathcal{P}_{jk}$ must consist of one edge $(j,k)$ as $\ell(\mathcal{C}') > \ell(\mathcal{C})$.
	\par
	Let us consider the expression of $\bm{x}$ of the form
	\begin{align*}
		\bm{x} = \bigoplus_{\lambda' \in \Lambda} \bm{y}_{\lambda'}, \quad \bm{y}_{\lambda'} \in W(A,\lambda').
	\end{align*}
	There exists $\mu \in \Lambda$ such that $x_{i} = [\bm{y}_{\mu}]_{i}$.
	We will prove that $\mu = \lambda$.
	\par
	First, suppose that $\bm{y} \in W(A,\mu), \bm{x} \geq \bm{y}$ and $x_{i} = y_{i}$
	for some algebraic eigenvalue $\mu > \lambda$.
	Since $i \not\in V(\mathcal{C})$, we have
	\begin{align*}
		\lambda \otimes x_{i} = [A \otimes \bm{x}]_{i} \geq [A \otimes \bm{y}]_{i}
		\geq \mu \otimes y_{i} = \mu \otimes x_{i} > \lambda \otimes x_{i},
	\end{align*}
	 which is a contradiction.
	\par
	Next, suppose that $\bm{y} \in W(A,\mu), \bm{x} \geq \bm{y}$ and $x_{i} = y_{i}$
	for some algebraic eigenvalue $\mu < \lambda$.
	As noted above, either $j \not\in V(\mathcal{C})$ or $(j,k) \in E(\mathcal{C})$ occurs.
	By Lemma \ref{indepcond} below, we have the following facts corresponding to these cases.
	\begin{enumerate}
		\item 
		If $j \not\in V(\mathcal{C})$, then the maximum is attained exactly twice in the $j$th row of 
		$\begin{pmatrix} A & \lambda \otimes E_{n} \end{pmatrix} \otimes \begin{pmatrix} \bm{x} \\ \bm{x} \end{pmatrix}$.
		\item 
		If $(j,k) \in E(\mathcal{C})$, then the maximum is attained exactly twice in both the $i$th and the $j$th row of 
		$\begin{pmatrix} A & \lambda \otimes E_{n} \end{pmatrix} \otimes \begin{pmatrix} \bm{x} \\ \bm{x} \end{pmatrix}$.
	\end{enumerate}
	If $j \not\in V(\mathcal{C})$, then we have
	\begin{align*}
		[A \otimes \bm{x}]_{j} = [(A_{\mathcal{C}} \oplus \lambda \otimes E_{\setminus\mathcal{C}} ) \otimes \bm{x}]_{j} 
		= \lambda \otimes x_{j} = a_{ji} \otimes x_{i}
	\end{align*}
	since $(j,i) \in E(\mathcal{C}')$.
	The maximum in the $j$th row 
	of $\begin{pmatrix} A & \lambda \otimes E_{n} \end{pmatrix} \otimes \begin{pmatrix} \bm{x} \\ \bm{x} \end{pmatrix}$
	is attained only with these two terms.
	On the other hand, from our assumption, we have
	\begin{align*}
		&[A \otimes \bm{y}]_{j} \geq a_{ji} \otimes y_{i} = a_{ji} \otimes x_{i}, \\
		&a_{ji'} \otimes y_{i'} \leq a_{ji'} \otimes x_{i'} < a_{ji} \otimes x_{i} \quad \text{ for all } i' \neq i, \text{ and }\\
 		&\mu \otimes y_{j} < \lambda \otimes x_{j} = a_{ji} \otimes x_{i}.
	\end{align*}
	This shows that the maximum is attained exactly once in the $j$th row of 
	$\begin{pmatrix} A & \mu \otimes E_{n} \end{pmatrix} \otimes \begin{pmatrix} \bm{y} \\ \bm{y} \end{pmatrix}$,
	which contradicts the fact that $\bm{y} \in W(A,\mu)$.
	\par
	If $(j,k) \in E(\mathcal{C})$, we have
	\begin{align*}
		[A \otimes \bm{x}]_{j} = a_{jk} \otimes x_{k} = a_{ji} \otimes x_{i} 
	\end{align*}
	since $(j,i) \in E(\mathcal{C}')$.
	The maximum in the $j$th row
	of $\begin{pmatrix} A & \lambda \otimes E_{n} \end{pmatrix} \otimes \begin{pmatrix} \bm{x} \\ \bm{x} \end{pmatrix}$
	is attained only with these two terms.
	Since 
	\begin{align*}
		&[A \otimes \bm{y}]_{j} \geq a_{ji} \otimes y_{i} = a_{ji} \otimes x_{i}, \\
		&a_{ji'} \otimes y_{i'} \leq a_{ji'} \otimes x_{i'} < a_{ji} \otimes x_{i} \quad \text{ for all } i' \neq i,k,  \text{ and }\\
 		&\mu \otimes y_{j} < \lambda \otimes x_{j} = a_{ji} \otimes x_{i},
	\end{align*}
	and $\bm{y} \in W(A,\mu)$, we must have
	\begin{align*}
		[A \otimes \bm{y}]_{j} = a_{ji} \otimes y_{i} = a_{jk} \otimes y_{k} = a_{ji} \otimes x_{i} = a_{jk} \otimes x_{k},
	\end{align*}
	which leads to $x_{k} = y_{k}$.
	Then, for the $i$th row we similarly have
	\begin{align*}
 		&[A \otimes \bm{x}]_{i} = \lambda \otimes x_{i} = a_{ik} \otimes x_{k}, \\
		&[A \otimes \bm{y}]_{i} \geq a_{ik} \otimes y_{k} = a_{ik} \otimes x_{k}, \\
		&a_{ik'} \otimes y_{k'} \leq a_{ik'} \otimes x_{k'} < a_{ik} \otimes x_{k} \quad \text{ for all } k' \neq k, \text{ and }\\
		&\mu \otimes y_{i} < \lambda \otimes x_{i}. 
	\end{align*}
	This shows that the maximum is attained exactly once in the $i$th row of 
	$\begin{pmatrix} A & \mu \otimes E_{n} \end{pmatrix} \otimes \begin{pmatrix} \bm{y} \\ \bm{y} \end{pmatrix}$,
	which contradicts the fact that $\bm{y} \in W(A,\mu)$.
	\par
	Thus, we have proved that 	if $\bm{y} \in W(A,\mu), \bm{x} \geq \bm{y}$ and $x_{i} = y_{i}$
	for some algebraic eigenvalue $\mu \in \Lambda$, then $\mu = \lambda$.
	Further, we have $\bm{x} = \bm{y}$ in this case.
	Indeed, $(-x_{i}) \otimes \bm{x}$ is the minimal element in
	\begin{align*}
		W(A,\lambda)^{(i)} = \{ (-v_{i}) \otimes \bm{v} \ |\ \bm{v} \in W(A,\lambda), v_{i} \neq \varepsilon \}
	\end{align*}
	by Lemma \ref{minimalelem}.
	This means that $\bm{x}$ cannot be expressed as the max-plus linear combination of
	other algebraic eigenvectors in $\bigcup_{\lambda \in \Lambda} \mathcal{B}_{\lambda}$.
\end{proof}
\begin{lemma} \label{indepcond}
	Under the setting in Theorem \ref{main1} and its proof, the following facts hold.
	\begin{enumerate}
		\item 
		If $j \not\in V(\mathcal{C})$, then the maximum is attained exactly twice in the $j$th row of 
		$\begin{pmatrix} A & \lambda \otimes E_{n} \end{pmatrix} \otimes \begin{pmatrix} \bm{x} \\ \bm{x} \end{pmatrix}$.
		\item 
		If $(j,k) \in E(\mathcal{C})$, then the maximum is attained exactly twice in both the $i$th and the $j$th row of 
		$\begin{pmatrix} A & \lambda \otimes E_{n} \end{pmatrix} \otimes \begin{pmatrix} \bm{x} \\ \bm{x} \end{pmatrix}$.
	\end{enumerate}
\end{lemma}
\begin{proof}
	We first consider Case 1.
	By Proposition \ref{balc-adj}, the $j$th column of $(B_{A,\lambda,\mathcal{C}})^{*}$ is a scalar multiple of 
	the $j$th column of $\Gamma(A,\lambda) = \mathrm{adj} (A \oplus \lambda \otimes E_{n})$ since $j \not\in V(\mathcal{C})$.
	Further, since the veritces $i$ and $j$ are on the same critical circuit $D$ in $G(B_{A,\lambda,\mathcal{C}})$,
	$\bm{x}$ is a scalar multiple of the $j$th column of $\Gamma(A,\lambda)$.
	Suppose that the maximum is attained with more than three terms in the $j$th row of 
	$\begin{pmatrix} A & \lambda \otimes E_{n} \end{pmatrix} \otimes \begin{pmatrix} \bm{x} \\ \bm{x} \end{pmatrix}$.
	Then, the maximum is attained with $a_{jp} \otimes x_{p}$ in addition to $\lambda \otimes x_{j}$ and $a_{ji} \otimes x_{i}$.
	This means that the maximum in the product of the $j$th row of $A \oplus \lambda \otimes E_{n}$
	and the $j$th column of $\Gamma(A,\lambda)$ is attained with the $p$th term.
	By Lemma \ref{adjproduct}, there is a permutation $\sigma \in S_{n}$ attaining $\det(A \oplus \lambda \otimes E_{n})$
	such that $\sigma(j) = p$.
	This permutation $\sigma$ corresponds to a $\lambda$-maximal multi-circuit $\mathcal{C}''$ in $G(A)$ containing the edge $(j,p)$.
	Then, $\psi_{\mathcal{C}}(\mathcal{C}'')$ contains a critical circuit that intersects $D$ at the vertex $j$
	in the graph $G(B_{A,\lambda,\mathcal{C}})$, which contradicts the assumption for $A$ in Theorem \ref{main1}.
	\par
	A similar argument can be applied to Case 2.
	In this case, $D = (i,k,i)$ is a critical circuit in $G(B_{A,\lambda,\mathcal{C}})$ by the construction of $\psi_{\mathcal{C}}$.
	Since $i \not\in V(\mathcal{C})$ and $k$ is the succeeding vertex of $j$ in $\mathcal{C}$,
	the $i$th and the $k$th columns of $(B_{A,\lambda,\mathcal{C}})^{*}$ are scalar multiples of 
	the $i$th and the $j$th columns of $\Gamma(A,\lambda)$, respectively.
	Hence, all of these columns are scalar multiples of $\bm{x}$.
	Suppose that in the $i$th row of 
	$\begin{pmatrix} A & \lambda \otimes E_{n} \end{pmatrix} \otimes \begin{pmatrix} \bm{x} \\ \bm{x} \end{pmatrix}$
	the maximum is attained with $a_{ip} \otimes x_{p}$ in addition to $\lambda \otimes x_{i}$ and $a_{ik} \otimes x_{k}$.
	Then, there is a $\lambda$-maximal multi-circuit $\mathcal{C}''$ in $G(A)$ containing the edge $(i,p)$.
	Hence, $\psi_{\mathcal{C}}(\mathcal{C}'')$ contains a critical circuit that intersects $D$ at the vertex $i$ 
	in $G(B_{A,\lambda,\mathcal{C}})$, which contradicts to the assumption for $A$ in Theorem \ref{main1}.
	Next suppose that in the $j$th row of 
	$\begin{pmatrix} A & \lambda \otimes E_{n} \end{pmatrix} \otimes \begin{pmatrix} \bm{x} \\ \bm{x} \end{pmatrix}$
	the maximum is attained with $a_{jp} \otimes x_{p}$ or $\lambda \otimes x_{j}$ 
	in addition to $a_{ji} \otimes x_{i}$ and $a_{jk} \otimes x_{k}$.
	In the former case, there is a $\lambda$-maximal multi-circuit $\mathcal{C}''$ in $G(A)$ containing the edge $(j,p)$.
	Then, $\psi_{\mathcal{C}}(\mathcal{C}'')$ contains a critical circuit with edge $(k,p)$, which intersects $D$ at the vertex $k$.
	In the latter case, there is a $\lambda$-maximal multi-circuit $\mathcal{C}''$ in $G(A)$ not containing the vertex $j$.
	Then, $\psi_{\mathcal{C}}(\mathcal{C}'')$ contains the edge $(k,j)$ and intersects $D$ at the vertex $k$.
	Hence, both cases contradict the assumption for $A$ in Theorem \ref{main1}.
\end{proof}
\begin{example} \label{indepexm}
	We again consider a matrix
	\begin{align*}
		A = \begin{pmatrix}
		10 & 10 & 9 & \varepsilon \\
		9 & 1 & \varepsilon & \varepsilon \\
		\varepsilon & \varepsilon & \varepsilon & 9 \\
		9 & \varepsilon & \varepsilon & \varepsilon 
		\end{pmatrix},
	\end{align*}
	whose algebraic eigenvalues are $10, 9, 8, 1$ and corresponding algebraic eigenvectors are
	\begin{align*}
		\bm{x}_{1} = {}^t\!(2,1,0,1), \quad \bm{x}_{2} = {}^t\!(0,0,0,0), \quad
		 \bm{x}_{3} = {}^t\!(0,1,2,1), \quad \bm{x}_{4} =  {}^t\!(0,15,16,8),
	\end{align*}
	respectively. 
	The $\lambda$-maximal multi-circuits are
	\begin{align*}
		&\mathcal{C}_{1} := \emptyset \text{ and } \mathcal{C}_{2} := \{(1,1)\} & &\text{ when } \lambda = 10, \\
		&\mathcal{C}_{2} \text{ and } \mathcal{C}_{3} := \{(1,2,1)\} & &\text{ when } \lambda = 9, \\
		&\mathcal{C}_{3} \text{ and } \mathcal{C}_{4} := \{(1,3,4,1)\} & &\text{ when } \lambda = 8, \text{ and }\\
		&\mathcal{C}_{4} \text{ and } \mathcal{C}_{5} := \{(1,3,4,1),(2,2)\} & &\text{ when } \lambda = 1.
	\end{align*}
	The critical circuits in $G(B_{A,10,\mathcal{C}_{1}}), G(B_{A,9,\mathcal{C}_{2}}), G(B_{A,8,\mathcal{C}_{3}})$
	and $G(B_{A,1,\mathcal{C}_{4}})$ are $(1,1), (1,2,1), (1,2,3,4,1)$ and $(2,2)$, respectively, see Figure \ref{indepexmfig}.
	Since the assumption for $A$ in Theorem \ref{main1} obviously holds, $\{\bm{x}_{1}, \bm{x}_{2}, \bm{x}_{3}, \bm{x}_{4}\}$
	is a basis of $W(A,10) \oplus W(A,9) \oplus W(A,8) \oplus W(A,1)$.
	\vspace*{5mm}
	\begin{figure}[htbp] \label{indepexmfig}
		\begin{center}
		\includegraphics[width=0.9\textwidth]{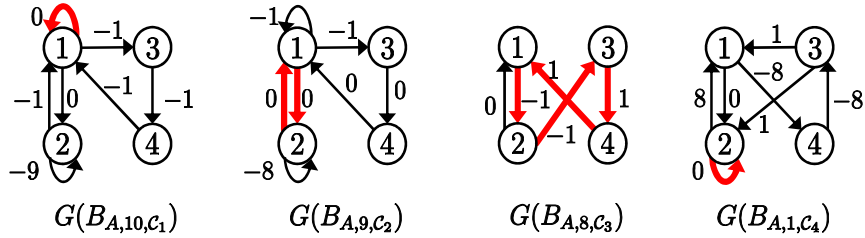}
		\caption{Graphs $G(B_{A,\lambda,\mathcal{C}})$ in Example \ref{indepexm}. 
		Bold red arrows indicate critical circuits.}
		\end{center}
	\end{figure}	
\end{example}
\section{Orthogonality of algebraic eigenvectors}
In this section, we focus on max-plus symmetric matrices.
A square matrix $A = (a_{ij}) \in \mathbb{R}_{\max}^{n \times n}$ is called symmetric if $a_{ij} = a_{ji}$ for all $i,j$.
In the conventional linear algebra, two eigenvectors of a symmetric matrix
with respect to distinct eigenvalues are orthogonal to each other.
We will investigate the max-plus analogue of this fact.
In the tropical geometric sense, two vectors $\bm{x} = {}^t\!(x_{1},x_{2},\dots,x_{n}) \in \mathbb{R}_{\max}^{n}$ and 
$\bm{y} = {}^t\!(y_{1},y_{2},\dots,y_{n}) \in \mathbb{R}_{\max}^{n}$ are called orthogonal to each other
if the maximum is attained at least twice in the inner product ${}^t\!\bm{x} \otimes \bm{y} = \bigoplus_{i=1}^{n} x_{i} \otimes y_{i}$.
\begin{theorem}
	Any two algebraic eigenvectors of a max-plus symmetric matrix with respect to distinct algebraic eigenvalues are
	orthogonal to each other.
\end{theorem}
\begin{proof}
	Let $A = (a_{ij}) \in \mathbb{R}_{\max}^{n \times n}$ be a symmetric matrix and 
	$\lambda$ and $\mu$ be algebraic eigenvalues of $A$ with $\lambda > \mu$.
	Take $\bm{x} \in W(A,\lambda)$ and $\bm{y} \in W(A,\nu)$.
	Note that if both $\bm{v}_{1}$ and $\bm{v}_{2}$ are orthogonal to $\bm{u}$ 
	then so is $\alpha_{1} \otimes \bm{v}_{1} \oplus \alpha_{2} \otimes \bm{v}_{2}$ for $\alpha_{1},\alpha_{2} \in \mathbb{R}$.
	Hence, we may assume that $\bm{x}$ and $\bm{y}$ are in bases of $W(A,\lambda)$ and $W(A,\mu)$, respectively.
	\par
	Since $A \otimes \bm{x} = (A \oplus \lambda \otimes E_{n}) \otimes \bm{x} \geq \lambda \otimes \bm{x}$,
	we first observe that
	\begin{align*}
		{}^t\!\bm{x} \otimes A \otimes \bm{y} = \bigoplus_{i,j} x_{i} \otimes a_{ij} \otimes y_{j} 
			\geq \lambda \otimes {}^t\!\bm{x} \otimes \bm{y}.
	\end{align*}
	We consider the case ${}^t\!\bm{x} \otimes \bm{y} \neq \varepsilon$; otherwise the statement of the theorem is obvious.
	We define the set 
	\begin{align*}
	\Omega = \{(i,j) \ |\ {}^t\!\bm{x} \otimes A \otimes \bm{y} = x_{i} \otimes a_{ij} \otimes y_{j} \neq \varepsilon  \}.
	\end{align*}
	We also define the subsets $I$ and $J$ of indices by
	\begin{align*}
		I = \{ i \ |\ (i,j) \in \Omega \text{ for some } j \}, \quad J = \{ j \ |\ (i,j) \in \Omega \text{ for some } i \}.
	\end{align*}
	For $i \in I$, we have $[A \otimes \bm{y}]_{i} > \mu \otimes y_{i}$.
	Indeed, if $[A \otimes \bm{y}]_{i} \leq \mu \otimes y_{i}$, then for $j$ with $(i,j) \in \Omega$
	we have 
	\begin{align*}
		\mu \otimes x_{i} \otimes y_{j} 
		\geq x_{i} \otimes [A \otimes \bm{y}]_{i} 
		\geq x_{i} \otimes a_{ij} \otimes y_{j} 
		= {}^t\!\bm{x} \otimes A \otimes \bm{y}
		\geq \lambda \otimes {}^t\!\bm{x} \otimes \bm{y},
	\end{align*}
	contradicting the assumption $\lambda > \mu$.
	\par
	We will prove that $|I| < |J|$.
	On the contrary, suppose $|I| \geq |J|$.
	Take $\bm{\zeta} \in Z$ and consider the matrix $A(\bm{\zeta};\delta)$, see Section 2.5.
	In particular, we take $\bm{\zeta}$ so that each minor of $A(\bm{\zeta};\delta)$ is attained with 
	exactly one permutation or is $\varepsilon$.
	Since $\bm{y}$ is in a basis of $W(A,\mu)$, 
	it is a scalar multiple of some column of $\Gamma(A,\mu)$ if $\mu \neq \varepsilon$.
	Hence, there are $\mu(\bm{\zeta};\delta) \in \mathcal{B}(\mu,n\delta)$ and 
	$\bm{y}(\delta) \in W(A(\bm{\zeta};\delta), \mu(\bm{\zeta};\delta), \mathcal{C}(\bm{\zeta};\delta))$ for $\delta > 0$
	such that $\lim_{\delta \to +0} \bm{y}(\delta) = \bm{y}$.
	As $\mathcal{B}(\varepsilon,n\delta) = \{\varepsilon\}$,
	this also holds in the case with $\mu = \varepsilon$ by letting $\mu(\bm{\zeta};\delta) = \varepsilon$.
	For $i \in I$, the inequality $[A \otimes \bm{y}]_{i} > \mu \otimes y_{i}$ above implies
	\begin{align*}
		[(A(\bm{\zeta};\delta) \oplus \mu(\bm{\zeta};\delta) \otimes E_{n}) \otimes \bm{y}(\delta)]_{i}
		> \mu(\bm{\zeta};\delta) \otimes y_{i}(\delta)
	\end{align*}
	by considering the limit $\delta \to +0$.
	Since $\bm{y}(\delta) \in W(A(\bm{\zeta};\delta), \mu(\bm{\zeta};\delta), \mathcal{C}(\bm{\zeta};\delta))$,
	for each $i \in I$ there exists at least two indices $k$ such that 
	\begin{align*}
		[(A(\bm{\zeta};\delta) \oplus \mu(\bm{\zeta};\delta) \otimes E_{n}) \otimes \bm{y}(\delta)]_{i}
		= (a_{ik} - \zeta_{ik}\delta) \otimes y_{k}(\delta). 
	\end{align*}	
	For such $k$, taking the limit $\delta \to +0$, we have
	\begin{align*}
		&[(A \oplus \mu \otimes E_{n}) \otimes \bm{y}]_{i} = a_{ik} \otimes y_{k} \\
		\Longrightarrow \qquad 
		& {}^t\!\bm{x} \otimes A \otimes \bm{y} = x_{i} \otimes [A \otimes \bm{y}]_{i} = x_{i} \otimes a_{ik} \otimes y_{k} \\
		\Longrightarrow \qquad
		&(i,k) \in \Omega \\
		\Longrightarrow \qquad
		&k \in J.
	\end{align*}
	Hence, the tropical kernel of  $[A(\bm{\zeta},\delta)]_{IJ}$ has a nontrivial vector $[\bm{y}(\delta)]_{J}$,
	where $[*]_{IJ}$ is the submatrix whose rows and columns are indexed by $I$ and $J$, respectively,
	and $[\bm{y}]_{J}$ is the restriction of $\bm{y}$ to the entries indexed by $J$.
	Since $|I| \geq |J|$, there is a $|J| \times |J|$ singular submatrix by Proposition \ref{ker},
	which contradicts the choice of $\bm{\zeta} \in Z$.
	Thus, we have proved that $|I| < |J|$.
	\par
	Next, we show that there exist at least two $j \in J$ such that $[{}^t\!\bm{x} \otimes A]_{j} = \lambda \otimes x_{j}$.
	On the contrary, suppose that there is a subset $J' \subset J$ with $|J'| = |J| - 1$ such that
	$[{}^t\!\bm{x} \otimes A]_{j} > \lambda \otimes x_{j}$ for all $j \in J'$.
	By the similar argument above, 
	 the tropical kernel of $[{}^t\!A(\bm{\zeta};\delta)]_{J'I}$ has a nontrivial vector $[\bm{x}(\delta)]_{I}$.
	Since $|J'| = |J| - 1 \geq |I|$, this contradicts the choice of $\bm{\zeta} \in Z$.
	\par
	Taking two distinct indices $j_{1},j_{2} \in J$ such that $[{}^t\!\bm{x} \otimes A]_{j_{p}} = \lambda \otimes x_{j_{p}}, p=1,2$,
	we have
	\begin{align*}
		{}^t\!\bm{x} \otimes A \otimes \bm{y} = [{}^t\!\bm{x} \otimes A]_{j_{p}} \otimes y_{j_{p}} 
		= \lambda \otimes x_{j_{p}} \otimes y_{j_{p}}, \quad p=1,2.
	\end{align*}
	Hence, we have ${}^t\! \bm{x} \otimes \bm{y} = x_{j_{1}} \otimes y_{j_{1}} = x_{j_{2}} \otimes  y_{j_{2}}$, proving the orthogonality.
\end{proof}
\begin{example}
	We consider the symmetric matrix $A$ in Example \ref{algeigvecexm}:
	\begin{align*}
		A = \begin{pmatrix} 
			6 & 5 & 0 \\ 5 & 4 & \varepsilon \\ 0 & \varepsilon & 2
		\end{pmatrix}.
	\end{align*}
	The algebraic eigenvalues are $6,4,2$ and the corresponding algebraic eigenvectors are
	\begin{align*}
		\bm{x}_{1} = {}^t\!(0,-1,-6), \quad \bm{x}_{2} = {}^t\!(-1,0,-5), \quad \bm{x}_{3} = {}^t\!(-6,-5,0), 
	\end{align*}
	respectively.
	Then, we have
	\begin{align*}
		{}^t\!\bm{x}_{1} \otimes \bm{x}_{2} = (-1) \oplus (-1) \oplus (-11), \\
		{}^t\!\bm{x}_{1} \otimes \bm{x}_{3} = (-6) \oplus (-6) \oplus (-6), \\
		{}^t\!\bm{x}_{2} \otimes \bm{x}_{3} = (-7) \oplus (-5) \oplus (-5).
	\end{align*}
	Hence, any two of $\bm{x}_{1}, \bm{x}_{2}$ and $\bm{x}_{3}$ are orthogonal to each other.
\end{example}

\section{Proofs of results in Section 2.5}
In this section, we give proofs of propositions in Section 2.5.
For an algebraic eigenvalue $\lambda$ of $A \in \mathbb{R}_{\max}^{n \times n}$ and a $\lambda$-maximal multi-circuit $\mathcal{C}$,
we use the notation
\begin{align}
	W(A,\lambda,\mathcal{C}) =  \{ \bm{x} \in \mathbb{R}_{\max}^{n} \ |\ 
		(A_{\setminus\mathcal{C}} \oplus \lambda \otimes E_\mathcal{C}) \otimes \bm{x} 
		= (A_\mathcal{C} \oplus \lambda \otimes E_{\setminus\mathcal{C}}) \otimes \bm{x} \}
\end{align}
even when $A$ does not satisfy the assumption ($\bigstar$).
As noted in Section 2, the algebraic eigenspace of $A$ with respect to $\lambda$ is
\begin{align*}
	W(A,\lambda) = \bigcap_{\bm{\zeta} \in Z} \lim_{\delta \to +0} 
		\bigoplus_{p} W(A(\bm{\zeta};\delta), \lambda^{(p)}(\bm{\zeta};\delta), \mathcal{C}^{(p)}(\bm{\zeta};\delta)),
\end{align*}
where $\lambda^{(p)}(\bm{\zeta};\delta),\, p=1,2,\dots,$ are the (finite number of) roots of 
$\chi_{A(\bm{\zeta};\delta)}(t)$ in $\mathcal{B}(\lambda, n\delta)$.
\begin{proof}[Proof of Proposition \ref{subspace} for $\lambda \neq \varepsilon$]
	It is sufficient to prove that 
	\begin{align*}
		\tilde{W} := \lim_{\delta \to +0} \bigoplus_{p} 
		W(A(\bm{\zeta};\delta), \lambda^{(p)}(\bm{\zeta};\delta), \mathcal{C}^{(p)}(\bm{\zeta};\delta))
	\end{align*}
	is a subspace for any $\bm{\zeta} \in Z$.
	Let $\bm{g}_{i}$ be the $i$th column of $\Gamma(A,\lambda)$ and 
	$\bm{g}^{(p)}_{i}(\delta)$ be the $i$th column of $\Gamma(A(\bm{\zeta};\delta), \lambda^{(p)}(\bm{\zeta};\delta))$.
	When $\delta$ is sufficiently small,
	we can take $\mathcal{C}^{(p)}(\bm{\zeta};\delta)$ independent of $\delta$.
	A basis of $W(A(\bm{\zeta};\delta), \lambda^{(p)}(\bm{\zeta};\delta), \mathcal{C}^{(p)}(\bm{\zeta};\delta))$ can be taken as
	$\{ \bm{g}^{(p)}_{i}(\delta) \ |\ i \in I^{(p)} \}$ for some $I^{(p)} \subset \{1,2,\dots,n\}$ by Proposition \ref{balc} and \ref{balc-adj}.
	Now, we show that $\tilde{W}$ is a subspace spanned by $\{ \bm{g}_{i} \ |\ i \in \bigcup_{p} I^{(p)} \}$.
	\par
	Since entries of $\Gamma(A(\bm{\zeta};\delta), \lambda^{(p)}(\bm{\zeta};\delta))$ are continuous in $\delta$,
	we have 
	\begin{align*}
		\lim_{\delta \to +0} \Gamma(A(\bm{\zeta};\delta), \lambda^{(p)}(\bm{\zeta};\delta)) = \Gamma(A,\lambda),
	\end{align*}
	which yields $\bm{g}_{i} = \lim_{\delta \to +0} \bm{g}^{(p)}_{i}(\delta)$.
	Hence, the subspace spanned by $\{ \bm{g}_{i} \ |\ i \in \bigcup_{p} I^{(p)} \}$ is contained in $\tilde{W}$.
	\par
	Conversely, take any vector $\bm{x} \in \tilde{W}$. Then, there exist vectors
	 $\bm{x}(\delta) \in \bigoplus_{p} W(A(\bm{\zeta};\delta), \lambda^{(p)}(\bm{\zeta};\delta), \mathcal{C}^{(p)}(\bm{\zeta};\delta))$
	for $\delta > 0$ such that $\lim_{\delta \to +0} \bm{x}(\delta) = \bm{x}$.
	Without loss of generality, we may assume that
	\begin{align*}
		&\bm{x}(\delta) = \bigoplus_{p} \bigoplus_{i \in I^{(p)}} \alpha^{(p)}_{i}(\delta) \otimes \bm{g}^{(p)}_{i}(\delta), \\
		&\alpha^{(p)}_{i}(\delta) = \max\{ \alpha \in \mathbb{R}_{\max} \ |\ \alpha \otimes \bm{g}^{(p)}_{i}(\delta) \leq \bm{x}(\delta) \}.
	\end{align*}
	Since $\bm{g}_{i} = \lim_{\delta \to +0} \bm{g}^{(p)}_{i}(\delta)$, 
	the limit $\alpha_{i} := \lim_{\delta \to +0} \alpha^{(p)}_{i}(\delta)$ exists and 
	then $\bm{x} = \bigoplus_{i \in \bigcup_{p} I^{(p)}} \alpha_{i} \otimes \bm{g}_{i}$.
	\par
	Thus, we conclude that
	$\tilde{W}$ is a max-plus subspace of $\mathbb{R}_{\max}^{n}$ for all $\bm{\zeta} \in Z$ and so is $W(A,\lambda)$.
\end{proof}
Let $W'(A,\lambda)$ be the subspace spanned by the columns of $\Gamma(A,\lambda)$
corresponding to indices in valid equivalent classes, see Definition \ref{validclass}.
We will show that $W(A,\lambda) = W'(A,\lambda)$.
\begin{lemma} \label{shortlong}
	Let $\mathcal{C}_{s}$ and $\mathcal{C}_{l}$ be the shortest and the longest $\lambda$-maximal multi-circuits, respectively.
	Then, for each valid equivalent class $H$, there exists $j \in H$ such that $j \in V(\mathcal{C}_{l}) \setminus V(\mathcal{C}_{s})$.
\end{lemma}
\begin{proof}
	Since $H$ is valid, we can choose a $\lambda$-maximal mulit-circuit $\mathcal{C}_{1}$ and
	a critical circuit $D$ with $V(D) \subset \sigma_{\mathcal{C}_{1}}(H)$ in $G(B_{A,\lambda,\mathcal{C}_{1}})$
	such that the length of $\mathcal{C}_{2} := \varphi_{\mathcal{C}_{1}} (D)$ is larger than that of $\mathcal{C}_{1}$.
	Note that $V(\overleftarrow{D}) \subset \sigma_{\mathcal{C}_{2}}(H)$, 
	by the structure of the graph $G(B_{A,\lambda,\mathcal{C}_{2}})$ and the construction of the map $\varphi_{\mathcal{C}_{2}}$,
	where $\overleftarrow{D}$ is obtained from $D$ by reversing the directions of all edges.
	Let $\mathcal{D}^{\circ}_{s}$ be the collection of the circuits in $G(B_{A,\lambda,\mathcal{C}_{2}})$
	that are contained in $\psi_{\mathcal{C}_{2}} (\mathcal{C}_{s})$ and intersect $\overleftarrow{D}$.
	To prove $\mathcal{D}^{\circ}_{s} \neq \emptyset$,
	we show that the length of $\mathcal{C}^{\circ}_{s} := \varphi_{\mathcal{C}_{2}}(\mathcal{D}^{\circ}_{s})$ 
	is smaller than $\ell(\mathcal{C}_{2})$.
	On the contrary, suppose  $\ell(\mathcal{C}^{\circ}_{s}) \geq \ell(\mathcal{C}_{2})$. 
	The multi-circuit 
	$\mathcal{D}' := \psi_{\mathcal{C}_{2}} (\mathcal{C}_{s}) \setminus \mathcal{D}^{\circ}_{s} \cup \{ \overleftarrow{D} \}$
	has the weight $0$ in $G(B_{A,\lambda,\mathcal{C}_{2}})$.
	The length of $\mathcal{C}' := \varphi_{\mathcal{C}_{2}}(\mathcal{D}')$ would be smaller than $\ell(\mathcal{C}_{s})$
	because of the equality
	\begin{align*}
		\ell(\mathcal{C}') - \ell(\mathcal{C}_{2}) 
		= (\ell(\mathcal{C}_{s}) - \ell(\mathcal{C}_{2})) - (\ell(\mathcal{C}^{\circ}_{s}) - \ell(\mathcal{C}_{2})) 
			+ (\ell(\mathcal{C}_{1}) - \ell(\mathcal{C}_{2})), 
	\end{align*}
	which contradicts the minimality of the length of $\mathcal{C}_{s}$.
	Thus, since $\mathcal{D}^{\circ}_{s} = \emptyset$ implies 
	$\mathcal{C}^{\circ}_{s} = \varphi_{\mathcal{C}_{2}}(\emptyset) = \mathcal{C}_{2}$, 
	we have $\mathcal{D}^{\circ}_{s} \neq \emptyset$  and $V(\mathcal{D}^{\circ}_{s}) \subset \sigma_{\mathcal{C}_{2}}(H)$.	
	In this case, we also have $V(\overleftarrow{\mathcal{D}^{\circ}_{s}}) \subset \sigma_{\mathcal{C}_{s}}(H)$.
	Let $\mathcal{D}^{\circ}_{l}$ be the collection of the circuits in $G(B_{A,\lambda,\mathcal{C}_{s}})$
	that are contained in $\psi_{\mathcal{C}_{s}} (\mathcal{C}_{l})$ and intersect $\overleftarrow{\mathcal{D}^{\circ}_{s}}$.
	Then, the length of $\mathcal{C}^{\circ}_{l} := \varphi_{\mathcal{C}_{s}}(\mathcal{D}^{\circ}_{l})$ is larger than $\ell(\mathcal{C}_{s})$
	by the similar reason to above.
	Hence, we have $\emptyset \neq V(\mathcal{C}^{\circ}_{l}) \setminus V(\mathcal{C}_{s})
	\subset \sigma_{\mathcal{C}_{s}}^{-1}(V(\mathcal{D}^{\circ}_{l})) \subset H$.
	Since $V(\mathcal{C}^{\circ}_{l}) \setminus V(\mathcal{C}_{s}) \subset V(\mathcal{C}_{l}) \setminus V(\mathcal{C}_{s})$,
	we can find $j \in H$ such that $j \in V(\mathcal{C}_{l}) \setminus V(\mathcal{C}_{s})$.
\end{proof}
\begin{proof}[Proof of Proposition \ref{modifeigspace}]
	We will show the equality
	\begin{align*}
		W'(A,\lambda) = \lim_{\delta \to +0} \bigoplus_{p} 
		W(A(\bm{\zeta};\delta), \lambda^{(p)}(\bm{\zeta};\delta), \mathcal{C}^{(p)}(\bm{\zeta};\delta))
	\end{align*}
	for all $\bm{\zeta} \in Z$.
	We denote the right-hand side of the above equality by $\tilde{W}$.
	It is sufficient to prove that the set of columns, or equivalence classes, of $\Gamma(A,\lambda)$ 
	spanning $W'(A,\lambda)$ and $\tilde{W}$ are identical.
	\par
	First, consider the $i$th column of $\Gamma(A,\lambda)$ that is in $\tilde{W}$.
	Taking sufficiently small $\delta > 0$, we see that the $i$th column of 
	$\Gamma(A(\bm{\zeta};\delta), \lambda^{(p)}(\bm{\zeta};\delta))$ belongs to 
	$W(A(\bm{\zeta};\delta), \lambda^{(p)}(\bm{\zeta};\delta), \mathcal{C}^{(p)}(\bm{\zeta};\delta))$ for some $p$.
	Since $A(\bm{\zeta};\delta)$ satisfies ($\bigstar$), 
	there exist $\lambda^{(p)}(\bm{\zeta};\delta)$-maximal multi-circuits $\mathcal{C}_{1}$ and $\mathcal{C}_{2}$ 
	in $G(A(\bm{\zeta};\delta))$ with $\ell(\mathcal{C}_{1}) \neq \ell(\mathcal{C}_{2})$ 
	such that $i \in V(\mathcal{C}_{1}) \setminus V(\mathcal{C}_{2})$. 
	Both $\mathcal{C}_{1}$ and $\mathcal{C}_{2}$ are also $\lambda$-maximal multi-circuits in $G(A)$.
	Hence, $i$ is in a valid equivalence class.
	\par
	Conversely, let $H$ be a valid equivalence class.
	Among the roots of $\chi_{A(\bm{\zeta};\delta)}(t)$ in $\mathcal{B}(\lambda, n\delta)$,
	let $\lambda^{(1)}(\bm{\zeta};\delta)$ and $\lambda^{(q)}(\bm{\zeta};\delta)$ 
	denote the maximum and the minimum ones, respectively.
	For sufficiently small $\delta$, the $\lambda^{(1)}(\bm{\zeta};\delta)$-maximal multi-circuit $\mathcal{C}_{s}$
	in $G(A(\bm{\zeta};\delta))$ with the minimum length is also an $\lambda$-maximal multi-circuit in $G(A)$ with the minimum length.
	Similarly, the $\lambda^{(q)}(\bm{\zeta};\delta)$-maximal multi-circuit $\mathcal{C}_{l}$ in $G(A(\bm{\zeta};\delta))$
	with the maximum length is also an $\lambda$-maximal multi-circuit in $G(A)$ with the maximum length.
	By Lemma \ref{shortlong}, there exists $j \in H \cap V(\mathcal{C}_{l}) \setminus V(\mathcal{C}_{s})$.
	For some $\lambda^{(p)}(\bm{\zeta};\delta) \in \mathcal{B}(\lambda, n\delta)$,
	there are two $\lambda^{(p)}(\bm{\zeta};\delta)$-maximal multi-circuits $\mathcal{C}_{1}$ and $\mathcal{C}_{2}$ 
	with $j \in V(\mathcal{C}_{2}) \setminus V(\mathcal{C}_{1})$.
	Hence, the $j$th column of $\Gamma(A(\bm{\zeta};\delta), \lambda^{(p)}(\bm{\zeta};\delta))$ belongs to 
	$W(A(\bm{\zeta};\delta), \lambda^{(p)}(\bm{\zeta};\delta), \mathcal{C}^{(p)}(\bm{\zeta};\delta))$.
	By the proof of Proposition \ref{subspace}, the $j$th column of $\Gamma(A,\lambda)$ is in $\tilde{W}$.
\end{proof}
In particular, the proof above shows that $\tilde{W}$ does not depend on the choice of $\bm{\zeta} \in Z$.
Hence, for $\lambda \neq \varepsilon$, we have 
\begin{align*}
	W(A, \lambda) = \lim_{\delta \to +0} \bigoplus_{p} 
		W(A(\bm{\zeta};\delta), \lambda^{(p)}(\bm{\zeta};\delta), \mathcal{C}^{(p)}(\bm{\zeta};\delta))
\end{align*}
for any $\bm{\zeta} \in Z$.
\begin{proof}[Proof of Proposition \ref{ifonlyif} for $\lambda \neq \varepsilon$]
	First, $\lambda \neq \varepsilon$ is a root of $\chi_{A}(t)$ if and only if there are two $\lambda$-maximal multi-circuits
	with the different lengths.
	On the other hand, for columns of $\Gamma(A,\lambda)$, there is a valid equivalent class
	if and only if there are two $\lambda$-maximal multi-circuits with the different lengths.
	Hence, $\lambda \neq \varepsilon$ is a root of $\chi_{A}(t)$ if and only if $W(A,\lambda) = W'(A,\lambda) \neq \{\mathcal{E}\}$.
\end{proof}
\begin{proof}[Proof of Proposition \ref{dimsub2} for $\lambda \neq \varepsilon$]
	By Proposition \ref{dimsub},
	the dimension of $W(A(\bm{\zeta};\delta), \lambda^{(p)}(\bm{\zeta};\delta), \mathcal{C}^{(p)}(\bm{\zeta};\delta))$
	cannot exceed the multiplicity of the root $\lambda^{(p)}(\bm{\zeta};\delta)$ of $\chi_{A(\bm{\zeta};\delta)}(t)$.
	Summing up the dimensions of the subspaces and the multiplicities for all $p$, respectively,
	we prove Proposition \ref{dimsub2}.
\end{proof}
\begin{proof}[Proof of Proposition \ref{generalize} for $\lambda \neq \varepsilon$]
	If $A$ satisfies ($\bigstar$), then we can take $\zeta_{ij} = 0$ for all $i,j$.
	Then, 
	\begin{align*}
		W(A,\lambda)
		=W(A(\bm{\zeta};\delta), \lambda^{(p)}(\bm{\zeta};\delta), \mathcal{C}^{(p)}(\bm{\zeta};\delta)) 
		= W(A,\lambda,\mathcal{C})
	\end{align*}
	for a $\lambda$-maximal multi-circuit $\mathcal{C}$.
\end{proof}
Next, we consider the case with $\lambda = \varepsilon$.
Let $\mathcal{C}$ be a $\lambda$-maximal multi-circuit in $G(A)$.
Then, the equation defining $W(A,\varepsilon,\mathcal{C})$ can be written as a pair of two equations:
\begin{align*}
	[A_{\setminus\mathcal{C}}]_{KK} \otimes [\bm{x}]_{K} \oplus [A]_{KL} \otimes [\bm{x}]_{L}
		&= [A_\mathcal{C}]_{KK} \otimes [\bm{x}]_{K}, \\ 
	[A]_{LK} \otimes [\bm{x}]_{K} \oplus [A]_{LL} \otimes [\bm{x}]_{L} &= \mathcal{E},
\end{align*}
where $K = V(\mathcal{C})$ and $L = \{1,2,\dots,n\} \setminus V(\mathcal{C})$.
If $A$ satisfies ($\bigstar$), then the graph $G([A_\mathcal{C}]_{KK}^{-1} \otimes [A_{\setminus\mathcal{C}}]_{KK})$
has no nonnegative circuit and hence the vector $\bm{x}$ satisfying the first equation is uniquely determined from $[\bm{x}]_{L}$ by
\begin{align*}
	[\bm{x}]_{K} = ([A_\mathcal{C}]_{KK}^{-1} \otimes [A_{\setminus\mathcal{C}}]_{KK})^{*} \otimes 
				([A_\mathcal{C}]_{KK}^{-1} \otimes [A]_{KL}) \otimes [\bm{x}]_{L}.
\end{align*}
Let $\xi_{A,\mathcal{C}}$ be the linear map $\mathbb{R}_{\max}^{n} \to \mathbb{R}_{\max}^{n}$ defined by
\begin{align*}
	[\xi_{A,\mathcal{C}}(\bm{x})]_{K} &= ([A_\mathcal{C}]_{KK}^{-1} \otimes [A_{\setminus\mathcal{C}}]_{KK})^{*} \otimes 
				([A_\mathcal{C}]_{KK}^{-1} \otimes [A]_{KL}) \otimes [\bm{x}]_{L}, \\
	[\xi_{A,\mathcal{C}}(\bm{x})]_{L} &= [\bm{x}]_{L}
\end{align*}
for all $\bm{x} \in \mathbb{R}_{\max}^{n}$.
The next lemma proves Proposition \ref{subspace} for $\lambda = \varepsilon$.
\begin{lemma} 
	For any $\bm{\zeta} \in Z$, let $\mathcal{C}(\bm{\zeta})$ be 
	 the $\varepsilon$-maximal multi-circuit in $G(A(\bm{\zeta};\delta))$ for sufficiently small $\delta > 0$.
	Then, $\lim_{\delta \to +0} W(A(\bm{\zeta};\delta), \varepsilon, \mathcal{C}(\bm{\zeta}))$
	is a subspace of $\mathbb{R}_{\max}^{n}$ spanned by 
	\begin{align*}
		\{ \xi_{A,\mathcal{C}(\bm{\zeta})}(\bm{e}_{j}) \ |\ 
			j \in L, [A \otimes \xi_{A,\mathcal{C}(\bm{\zeta})}(\bm{e}_{j})]_{L} = \mathcal{E} \},
	\end{align*}
	where $\bm{e}_{j}$ is the $j$th standard basis vector of $\mathbb{R}_{\max}^{n}$.
\end{lemma}
\begin{proof}
	By the above argument, it can be seen that $W(A(\bm{\zeta};\delta), \varepsilon, \mathcal{C}(\bm{\zeta}))$
	is spanned by 
	\begin{align*}
		\{ \xi_{A(\bm{\zeta};\delta),\mathcal{C}(\bm{\zeta})}(\bm{e}_{j}) \ |\ 
			j \in L, [A(\bm{\zeta};\delta) \otimes \xi_{A(\bm{\zeta};\delta),\mathcal{C}(\bm{\zeta})}(\bm{e}_{j})]_{L} 
			= \mathcal{E} \}
	\end{align*}
	for sufficiently small $\delta > 0$.
	Taking the limit $\delta \to +0$, we have proved the lemma.
\end{proof}
From this lemma, the subspace 
$\lim_{\delta \to +0} W(A(\bm{\zeta};\delta), \varepsilon, \mathcal{C}^{(p)}(\bm{\zeta};\delta))$ is determined 
depending only on $\mathcal{C}(\bm{\zeta})$. 
Hence, we denote it by $\tilde{W}(A,\mathcal{C}(\bm{\zeta}))$.
In particular, $W(A,\varepsilon)$ is an intersection of the finite number of max-plus subspaces.
Note that if $\bm{x} \in\tilde{W}(A,\mathcal{C}(\bm{\zeta}))$, then we have $\xi_{A,\mathcal{C}(\bm{\zeta})}(\bm{x}) = \bm{x}$.
\begin{proof}[Proof of Proposition \ref{ifonlyif} for $\lambda = \varepsilon$]
	Suppose $\lambda = \varepsilon$ is an algebraic eigenvalue of $A$.
	Then, consider the parametrized matrix $A(\mu) = A \oplus \mu \otimes E_{n}$.
	If $\mu$ is sufficiently close to $\varepsilon$, then $\mu$ is an algebraic eigenvalue of $A(\mu)$.
	Let $\bm{x}(\mu)$ be an algebraic eigenvector of $A(\mu)$ with respect to $\mu$.
	We may assume that the maximum entry of $\bm{x}(\mu)$ is $0$ 
	and each entry is of the form $\alpha + \beta \mu, \beta \geq 0$. 
	Then, $\lim_{\mu \to \varepsilon} \bm{x}(\mu) \in W(A,\varepsilon)$ and its maximum entry is $0$.
	\par
	Conversely, suppose $\lambda = \varepsilon$ is not an algebraic eigenvalue of $A$.
	This means that $G(A)$ has a multi-circuit with the length $n$, which becomes an $\varepsilon$-maximal multi-circuit.
	If there is a vector $\bm{x} \in W(A,\lambda) \setminus \{\mathcal{E} \}$, 
	then we have a nontrivial vector 
	$\bm{x}(\delta) \in W(A(\bm{\zeta};\delta),\varepsilon,\mathcal{C}(\bm{\zeta}))$ for some $\bm{\zeta} \in Z$
	and sufficiently small $\delta$.
	This implies $\bm{x}(\delta)$ belongs to the tropical kernel of $A(\bm{\zeta};\delta)$.
	Hence, $A(\bm{\zeta};\delta)$ is singular by Proposition \ref{ker}.  
	Then, we have two $\varepsilon$-maximal multi-circuits with the length $n$ in $G(A(\bm{\zeta};\delta))$,
	which contradicts the choice of $A(\bm{\zeta};\delta)$.
\end{proof}
\begin{lemma} \label{xiineq}
	Let $\bm{\zeta}_{1}, \bm{\zeta}_{2} \in Z$.
	If $\bm{x} \in \tilde{W}(A,\mathcal{C}(\bm{\zeta}_{1}))$, then $\xi_{A,\mathcal{C}(\bm{\zeta}_{2})}(\bm{x}) \geq \bm{x}$.
\end{lemma}
\begin{proof}
	Since $\tilde{W}(A,\mathcal{C}(\bm{\zeta}_{1})) \subset W(A,\varepsilon,\mathcal{C}(\bm{\zeta}_{1}))$,
	we have $A_{\setminus\mathcal{C}} \otimes \bm{x} = A_{\mathcal{C}} \otimes \bm{x}$
	for $\bm{x} \in \tilde{W}(A,\mathcal{C}(\bm{\zeta}_{1}))$.
	In particular, the maximum in $[A \otimes \bm{x}]_{i}$ is attained at least twice for each $i$.
	Hence, setting $K = V(\mathcal{C}(\bm{\zeta}_{2})), L = \{1,2,\dots,n\} \setminus V(\mathcal{C}(\bm{\zeta}_{2}))$, we have
	\begin{align*}
		[A_{\setminus\mathcal{C}(\bm{\zeta}_{2})}]_{KK} \otimes [\bm{x}]_{K} \oplus [A]_{KL} \otimes [\bm{x}]_{L}
			\geq [A_{\mathcal{C}(\bm{\zeta}_{2})}]_{KK} \otimes [\bm{x}]_{K},
	\end{align*}
	which yields 
	\begin{align*}
		[\xi_{A,\mathcal{C}(\bm{\zeta}_{2})}(\bm{x})]_{K} &= 
			([A_{\mathcal{C}(\bm{\zeta}_{2})}]_{KK}^{-1} \otimes [A_{\setminus\mathcal{C}(\bm{\zeta}_{2})}]_{KK})^{*} \otimes 
					([A_{\mathcal{C}(\bm{\zeta}_{2})}]_{KK}^{-1} \otimes [A]_{KL}) \otimes [\bm{x}]_{L} \\
			&\geq  [\bm{x}]_{K}
	\end{align*}
	by iteration.
\end{proof} 
\begin{proof}[Proof of Proposition \ref{dimsub2} for $\lambda = \varepsilon$]
	Take any vector $\bm{x} \in W(A,\varepsilon) \setminus \{\mathcal{E} \}$.
	Let $\mathcal{C}_{1}, \mathcal{C}_{2}, \dots, \mathcal{C}_{q}$ be the $\varepsilon$-maximal multi-circuits in $G(A)$.
	Since $\bm{x} \in \tilde{W}(A,\mathcal{C}_{1})$, we can express $\bm{x}$ as 
	\begin{align*}
		\bm{x} = \bigoplus_{j \not\in V(\mathcal{C}_{1})} \alpha_{j} \otimes \xi_{A,\mathcal{C}_{1}}(\bm{e}_{j}), 
		\quad \alpha_{j} \in \mathbb{R}_{\max}.
	\end{align*} 
	Note that if $\xi_{A,\mathcal{C}_{1}}(\bm{e}_{j}) \not\in \tilde{W}(A,\mathcal{C}_{1})$ then we set $\alpha_{j} = \varepsilon$.
	Let $\xi = \xi_{A,\mathcal{C}_{q}} \circ \cdots \circ \xi_{A,\mathcal{C}_{2}} \circ \xi_{A,\mathcal{C}_{1}}$
	and $X$ be the $n \times n$ matrix representing $\xi$.
	By Lemma \ref{xiineq}, we have $X \otimes \bm{e}_{j} \geq \bm{e}_{j}$.
	On the other hand, since $\bm{x} \in \tilde{W}(A,\mathcal{C}_{p})$ for $p=1,2,\dots,q$, we have $\xi(\bm{x}) = \bm{x}$.
	Hence, we have
	\begin{align*}
		\bm{x} = \xi^{n}(\bm{x}) = \bigoplus_{j \not\in V(\mathcal{C}_{1})} \alpha_{j} \otimes X^{\otimes n} \otimes \bm{e}_{j}.
	\end{align*} 
	If $\alpha_{j} \neq \varepsilon$, then the $j$th column of $X^{\otimes n}$ is identical to that of $X^{\otimes (n+1)}$,
	otherwise the graph $G(X)$ would have a circuit with the positive weight and hence
	some entry of $\alpha_{j} \otimes X^{\otimes k} \otimes \bm{e}_{j}$ would exceed that of $\bm{x}$ for sufficiently large $k$.
	Since $\xi_{A,\mathcal{C}_{p}} (\bm{y}) \in \tilde{W}(A,\mathcal{C}_{p})$ for any $\bm{y}$, 
	\begin{align*}
		X^{\otimes n} \otimes \bm{e}_{j} \leq \xi_{A,\mathcal{C}_{1}} (X^{\otimes n} \otimes \bm{e}_{j})
		\leq \cdots
		\leq X^{\otimes (n+1)} \otimes \bm{e}_{j} = X^{\otimes n} \otimes \bm{e}_{j}
	\end{align*}
	by Lemma \ref{xiineq}.
	Hence, we have $\xi_{A,\mathcal{C}_{p}} (X^{\otimes n} \otimes \bm{e}_{j}) = X^{\otimes n} \otimes \bm{e}_{j}$ for $p=1,2,\dots,q$.
	This means that $X^{\otimes n} \otimes \bm{e}_{j} \in \tilde{W}(A,\mathcal{C}_{p})$ for all $p = 1,2,\dots,q$ if $\alpha_{j} \neq \varepsilon$.
	Thus, $W(A,\varepsilon)$ is a subspace spanned by some of $X^{\otimes n} \otimes \bm{e}_{j}$ for $j \not\in V(\mathcal{C}_{1})$.
	Since $n - \ell(\mathcal{C}_{1})$ is the multiplicity of $\varepsilon$ in the characteristic polynomial of $A$, 
	we have proved the proposition.
\end{proof}
\begin{proof}[Proof of Proposition \ref{generalize} for $\lambda = \varepsilon$]
	If $A$ satisfies ($\bigstar$), then there exists exactly one $\varepsilon$-maximal multi-circuit in $G(A)$.
	Hence, it is sufficient to take $\zeta_{ij} = 0$ for all $i,j$.
\end{proof}

\section*{Acknowledgment}
This work was supported by the Grant-in-Aid for Early-Career Scientists No.~20K14367
 from the Japan Society for the Promotion of Science.

%\bibliographystyle{abbrv}
%\bibliography{reference}

\end{document}